\documentclass[11pt]{amsart}

\DeclareFontFamily{U}{matha}{\hyphenchar\font45}
\DeclareFontShape{U}{matha}{m}{n}{
  <5> <6> <7> <8> <9> <10> gen * matha
  <10.95> matha10 <12> <14.4> <17.28> <20.74> <24.88> matha12
  }{}
\DeclareSymbolFont{matha}{U}{matha}{m}{n}
\DeclareFontFamily{U}{mathx}{\hyphenchar\font45}
\DeclareFontShape{U}{mathx}{m}{n}{
  <5> <6> <7> <8> <9> <10>
  <10.95> <12> <14.4> <17.28> <20.74> <24.88>
  mathx10
  }{}
\DeclareSymbolFont{mathx}{U}{mathx}{m}{n}

\DeclareMathSymbol{\obot}         {2}{matha}{"6B}
\DeclareMathSymbol{\bigobot}       {1}{mathx}{"CB}

\usepackage[pdfauthor={Congling Qiu}, 
    pdftitle={???}%
  dvips,colorlinks=true]{hyperref}

\hypersetup{
    colorlinks,
    citecolor=black,
    filecolor=black,
    linkcolor=black,
    urlcolor=black
}

\usepackage[english]{babel}
\usepackage[toc,page]{appendix}
\usepackage{thm-restate}

\usepackage{comment}
  \usepackage{multirow}

\usepackage[utf8]{inputenc}

\usepackage{xtab}
\usepackage{amsmath, amssymb
}
\usepackage{mathrsfs}
\usepackage[all]{xy}
\usepackage{extarrows}
\usepackage{amsrefs}
\usepackage{tikz-cd}

\usepackage{enumerate}
\usepackage{mathtools,booktabs}
\usepackage{color}
\usepackage[nameinlink]{cleveref}

\usepackage{epstopdf} 
\usepackage{booktabs}

\numberwithin{equation}{section}

\theoremstyle{plain}
\newtheorem{proposition}{Proposition}[section]

\newtheorem{lem}[proposition]{Lemma}
\newtheorem{thm}[proposition]{Theorem}
\newtheorem{prop}[proposition]{Proposition}

\theoremstyle{definition}
\newtheorem{defn}[proposition]{Definition}

\newtheorem{eg}[proposition]{Example}

\theoremstyle{remark}
\newtheorem{rmk}[proposition]{Remark}

\numberwithin{equation}{section}

\newcommand{\BA}{{\mathbb {A}}} 
\newcommand{\BC}{{\mathbb {C}}} \newcommand{\BD}{{\mathbb {D}}}
 
 \newcommand{\BH}{{\mathbb {H}}}

 \newcommand{\BP}{{\mathbb {P}}}
\newcommand{\BQ}{{\mathbb {Q}}}

 \newcommand{\BZ}{{\mathbb {Z}}}

\newcommand{\Aut}{{\mathrm{Aut}}}

\newcommand{\Char}{{\mathrm{Char}}}

\newcommand{\End}{{\mathrm{End}}}

 \newcommand{\GL}{{\mathrm{GL}}}

\newcommand{\GSpin}{{\mathrm{GSpin}}}

\newcommand{\Jac}{{\mathrm{Jac}}}

 \newcommand{\Pic}{\mathrm{Pic}}

\newcommand{\rf}{{\mathrm{f}}}

\newcommand{\Sh}{{\mathrm{Sh}}}

\newcommand{\SL}{{\mathrm{SL}}}

 \newcommand{\CH}{{\mathrm{CH}}}

\newcommand{\Sym}{{\mathrm{Sym}}}

\newcommand{\tr}{{\mathrm{tr}}}

\newcommand{\Griff}{{\mathrm{Griff}}}
     
\def\sym{\mathrm{sym}}
\def\prp{\psi}
\def\CC{\mathbb{C}}
\def\QQ{\mathbb{Q}}
\def\ZZ{\mathbb{Z}}
\def\gks{\Delta_{GKS}}
\def\AJ{\mathrm{AJ}}
\def\Extm{\mathrm{Ext}_{\text{MHS}}}
\def\I{\mathsf{I}}
\def\J{\mathsf{J}}
\def\onto{\twoheadrightarrow}
\def\xt{X^{2}}
\def\xT{X^{3}}
\def\wpr{\bigwedge^3_{\mathrm{pr}}}
\def\ajeq{\underset{\AJ}{\equiv}}
\def\W{\mathcal{W}}
\usepackage{enumitem}

\newcommand\supervisor[1]{\def\@supervisor{#1}}

\newcounter{elno}

\newcommand{\into}{\hookrightarrow}
\renewcommand{\cong}{\simeq}

\usepackage{amsthm}
\usepackage{thmtools}


\author{Matt Kerr}

\author{Wanlin Li} 
\address{ }
\email{}

\author{Congling Qiu}

\author{Tonghai Yang}

\begin{document} 

\title[Ceresa and Gross--Kudla--Schoen cycles of modular curves]{Non-triviality of algebraic cycles associated to modular curves via pullback formula}

\begin{abstract}
   Associated to an algebraic curve $X$, there are two canonically constructed homologically trivial algebraic $1$-cycles, the \textit{Ceresa cycle} in 
 the Jacobian of $X$,  and the \textit{Gross--Kudla--Schoen modified diagonal cycle} in the triple product $X \times X \times X$. According to Shou-Wu Zhang (\cite{Zhang}), one is torsion if and only if the other is.  In this paper, we prove that these two cycles associated to a large family of modular curves are non-torsion in the corresponding Chow groups. We obtain the result by relating this problem to the
study of special cycles on orthogonal Shimura varieties. As the main ingredient and a result of independent interest, we develop a pullback formula for special divisors on modular curves embedded in their products via the diagonal map.
\end{abstract}

\maketitle

\section{Introduction}

\subsection{The Ceresa and Gross--Kudla--Schoen cycles}\label{subsec:MainCeresa}
The
\textit{ Ceresa cycle} $X^+_o - X^-_o$ is an algebraic $1$-cycle in the Jacobian of a curve $X$ with a fixed base point $o$. \footnote{Instead of a point, one can use a degree $1$ divisor and canonically can take $\xi \in \CH^1(X)$ satisfying $(2g-2)\xi=K_X$ where $K_X$ is the canonical class following \cite{Zhang}. See Example \ref{eg:diagonal} for a detailed discussion.} It is in the kernel of the cycle class map from the Chow group $\CH^{g-1}(\Jac(X))$ to a Weil cohomology theory $\textup{H}^{2g-2}(\Jac(X))$ and thus is referred to as homologically trivial. In \cite{Ceresa}, Ceresa proved this cycle is of infinite order modulo algebraic equivalence for a very general complex curve of genus $\ge 3$. From this result, the Ceresa cycle is viewed as one of the first examples showing the infinitude of the Griffiths group.

The method used in Ceresa's work \cite{Ceresa} involves degeneration on families of curves and thus it remains an interesting problem to determine the non-triviality of the Ceresa cycle for an explicit algebraic curve. When $X$ is hyperelliptic, the Ceresa cycle $X^+_o-X^-_o$ is linearly trivial if $o$ is chosen to be a Weierstrass point and algebraically trivial for any choice of base point. Hence it is a natural problem to study the relationship between the triviality of the Ceresa cycle $X^+_o - X^-_o$ and the hyperellipticity of $X$. There have been numerous works \cites{BS, LS, Qiu, QZ1,QZ2} devoted to showing triviality of the Ceresa cycle in $\CH^{g-1}(\Jac(C))$ or $\Griff^{g-1}(\Jac(C))$ for non-hyperelliptic curves\footnote{In this paper, we will work with Chow groups tensored with $\BQ$ and use notation $\CH^{g-1}(\Jac(C))$ to denote $\CH^{g-1}(\Jac(C))\otimes\BQ$. Thus, from now on, by ``trivial'' or ``nontrivial'', we will always mean the order of a cycle class in $\CH^{*}(\cdot)$ is finite or infinite and ignore the nontrivial torsion information contained in integral Chow groups.}.


In \cites{QZ1,QZ2}, Qiu and Zhang developed a sufficient condition for the Ceresa cycle to be trivial in the Chow group $\CH^{g-1}(\Jac(C))$, and used their criterion to construct explicit examples of non-hyperelliptic curves with trivial Ceresa cycle. In particular, in \cite{QZ2}, they applied their method to modular curves and Shimura curves but could only find a handful of such examples. This led Qiu to conjecture (in \cite{Qiu}*{Conjecture 1.2.2}) that only finitely many Shimura curves are associated with trivial Ceresa cycles. This statement can be seen as a generalization of the finiteness of the set of hyperelliptic Shimura curves.
Our first theorem shows that for a large class of modular curves, the associated Ceresa cycle is nontrivial in the Chow group $\CH^{g-1}(\Jac(C))$. 

\begin{restatable}{thm}{CeresaConcrete}
    \label{Ceresa-concrete} Given a positive integer $N$, let $\Gamma$ be the congruence subgroup $\Gamma_1(2N) \cap \Gamma(2)\subset \SL_2(\BZ)$ and let $X_N=\Gamma\backslash\overline{\mathbb{H}}$ be the modular curve associated to $\Gamma$.

(a) If $N$ satisfies one of the following conditions, then the Ceresa cycle associated to $X_N$ is nontrivial in the the Chow group $\CH^{g-1}(\Jac(X_N))$ with respect to any chosen base point:
\begin{enumerate}
    \item $N$ is divisible by a prime $p=37, 43, 53, 61, 67$ or $ p >71$,

    \item   $N$ is divisible by $p^2$ for some prime $ p \ge 11$.
\end{enumerate}

(b) When $N$ is big enough,  then the Ceresa cycle associated to $X_N$ is nontrivial in $\CH^{g-1}(\Jac(X_N))$ with respect to any chosen base point. For example 
$$N > 2^6 \cdot 3^4 \cdot 5^2 \cdot 7^2 \cdot \prod_{ \begin{matrix} 11 \le p \le 71 \\ p \ne 37, 43, 53, 61, 67 \end{matrix}} p$$
 is enough.   
\end{restatable}

More precisely, we obtain the following Theorem \ref{thm:main1} which implies Theorem \ref{Ceresa-concrete} when combined with results on non-vanishing of central derivative of Hecke $L$-functions.

\begin{thm}\label{thm:main1}
    Let $N$ be a positive integer such that there exists a weight $2$ normalized newform $f \in \mathcal{S}^{new}_2(\Gamma_0(N))^-$ satisfying $L'(f,1) \ne 0$. Let $\Gamma$ be the congruence subgroup $\Gamma_1(2N) \cap \Gamma(2)\subset \SL_2(\BZ)$ and let $X_N=\Gamma\backslash\overline{\mathbb{H}}$ be the modular curve associated to $\Gamma$. Then the Ceresa cycle associated to $X_N$ is nontrivial in the the Chow group $\CH^{g-1}(\Jac(X_N))$ with respect to any chosen base point.
\end{thm}

In the proof, we do not directly work with the Ceresa cycle, but rather with \textit{the Gross--Kudla--Schoen modified diagonal cycle} $\Delta_{GKS}(X,o)$ (\cite{GK}, \cite{GS}).  This is an algebraic $1$-cycle constructed from the pointed curve $(X,o)$ and lives in the triple product $X^3:=X \times X \times X$.  If we replace the base point $o$ by a degree $1$ divisor $\xi \in \Pic^1(X)$ satisfying $(2g-2)\xi=\omega_X$, it had been shown by S-W Zhang \cite[Theorem 1.5.5]{Zhang} that the Ceresa cycle $X^+_\xi-X^-_\xi$ is nontrivial in $\CH^{g-1}(\Jac(X))$ if and only if $\Delta_{GKS}(X,\xi)$ is nontrivial in $\CH^1(X^3)$. Moreover, by \cite[eq (6)]{DRS}, our proof shows 
 the same non-triviality statement as Theorem \ref{thm:main1} in terms of the $\Delta_{GKS}$ cycle.

\begin{restatable}{thm}{mainGKS}    \label{thm:main2}
    Let $N$ be a positive integer such that there exists a weight $2$ normalized newform $f \in \mathcal{S}^{new}_2(\Gamma_0(N))^-$ satisfying $L'(f,1) \ne 0$. Let $\Gamma$ be the congruence  subgroup $\Gamma_1(2N) \cap \Gamma(2)\subset\SL_2(\BZ)$ and let $X_N=\Gamma\backslash\overline{\mathbb{H}}$ be the modular curve associated to $\Gamma$. Then the Gross--Kudla--Schoen modified diagonal cycle associated to $X_N$ is of infinite order in the the Chow group $\CH^{2}(X_N^3)$ with respect to any chosen base point. 
\end{restatable}


    

\subsection{Special divisors on orthogonal Shimura varieties and a pullback formula}

The method to prove Theorem \ref{thm:main1} and Theorem \ref{thm:main2} was inspired by the work of Eskandari--Murty \cites{EM,EMPAMS} in which the authors showed the Ceresa cycle associated to the degree $p$ Fermat curve $F_p:x^p+y^p=z^p$ for $p>7$ prime is nontrivial in the Chow group for any chosen base point. Their work relies on the construction of the Chow--Heegner divisors in the work of Darmon--Rotger--Sols \cite{DRS}, in which the authors obtain a point $\Pi_Z$ in the Jacobian of a curve $X$ by choosing a correspondence $Z \in \CH^1(X \times X)$. The point $\Pi_Z \in \Jac(X)$ being non-torsion then implies that the cycle $X^+_o-X^-_o$ is nontrivial in $\CH^{g-1}(\Jac(X))$. We will review and discuss the construction of the Chow--Heegner divisor in Section \ref{sec:ChowHeegner}. 

In the theory of Shimura varieties, a modular curve $X$ and its self product $X^2=X\times X$ can be realized as Shimura varieties of orthogonal type. The diagonal embedding $\delta: X \hookrightarrow X^2$ can be obtained by a natural inclusion of quadratic spaces. There are families of natural special cycles on orthogonal Shimura varieties described in the work of Kudla \cite{Kud97}. The Heegner divisors on $X$ and the Hecke correspondences on $X^2$ arise this way as special divisors. Moreover, the Chow--Heegner divisors are obtained from pullbacks of special divisors on $X^2$ by $\delta^*$. In Section \ref{sec:explicitcover}, we discuss these constructions following Bruinier--Yang \cite{BY}. As was discussed in the last paragraph, the existence of any non-torsion Chow--Heegner divisor implies the non-triviality of the Ceresa cycle associated to the modular curve $X$. 

Thus, we raise a more general question in this setting. Namely, given an embedding of orthogonal type Shimura varieties $ \iota\colon \Sh(W)_K \hookrightarrow \Sh(V)_{\Tilde{K}}$ associated to quadratic spaces $W$ and $V$, which special cycles on $\Sh(W)_K$ can be obtained as pullbacks of special cycles on $\Sh(V)_{\Tilde{K}}$ via $\iota$?

In this paper, we discuss a special case of this problem by setting $\Sh(W)_K = X_N$ where $X_N$ is the modular curve defined in Theorem \ref{Ceresa-concrete} and $\Sh(V)_{\Tilde{K}} = X_N \times X_N$ with $\iota=\delta$ being the diagonal embedding. We obtain the following result, which shows all Heegner divisors \footnote{By Heegner divisor, we mean special divisors $Z_{X_N}(m_0,\mu_0)$ defined in Section \ref{subsec:specialdivisor}. Following Equation \eqref{eq:Zm0mu0} and Remark \ref{rmk:Heegner}, these are the pullback of Heegner divisors on $X_0(N)$.} on $X_N$ can be obtained from pullback.

\begin{thm}\label{tthm:modularpullback}
    Let $X_N$ be the modular curve defined in Theorem \ref{Ceresa-concrete} and $ \delta: X_N \hookrightarrow X_N\times X_N$ be the diagonal embedding. For any Heegner divisor $Z_{X_N}(m_0,\mu_0)$ on $X_N$, there exists a special divisor $Z \in \CH^1(X_N\times X_N)$ satisfying $Z_{X_N}(m_0,\mu_0)=\delta^*(Z)$.
\end{thm}


This result is proved in Section \ref{sec:explicitcover}.  We then obtain our Theorems \ref{thm:main1} and \ref{thm:main2} in Section \ref{sec:mainproof}, by combining Theorem \ref{tthm:modularpullback} with the existence of non-torsion Heegner divisors on $X_0(N)$ in the presence of a weight $2$ cusp form with nonvanishing $L'(f,1)$.

\subsection{History and future directions}\label{subsec:future}
The Ceresa and modified diagonal cycles have been the subject of several decades of study.  In \cite{Harris}, Harris numerically computed an invariant deduced from Hodge theory, the harmonic volume, to show the Ceresa cycle associated to the degree $4$ Fermat curve $F_4: x^4+y^4=z^4$ is algebraically nontrivial. In \cite{Bloch}, Bloch further proved that the Ceresa cycle is of infinite order in the Griffiths group for this particular genus $3$ curve $F_4$. The work of Harris and Bloch relied on the fact that the Jacobian of $F_4$ is isogenous to the triple self-product of an elliptic curve with complex multiplication by the imaginary quadratic field $\mathbb{Q}(i)$. 


The construction of Chow--Heegner divisors relies on correspondences $Z \in \CH^1(X \times X)$.
 When the curve $X$ is modular, every correspondence $Z$ is a linear combination of Hecke correspondences. In this setting,
the Chow--Heegner divisor construction has been used to construct explicit points on modular Jacobians and elliptic curves in \cites{DDLR,DRS}. As was discussed in \cite[Lemma 10]{DF}, for a given modular curve, we can pick a Hecke correspondence such that the divisor $\Pi_Z$ is supported on Heegner points. One would naturally hope to show that the Neron--Tate height of the divisor $\Pi_Z$ is nonzero and use this information to deduce the nontriviality of $\Delta_{GKS}$. But as mentioned in \cite[Page 419]{DF}, it is in general hard to assert the non-triviality of $\Pi_Z$ for $Z$ a Hecke correspondence. However, after an earlier version of this paper was posted online, Lupoian--Rawson \cite{LR} directly computed some Chow--Heegner divisors $\Pi_Z$ for modular curves $X_0(N)$ and correspondences including the Atkin--Lehner involution and Hecke correspondences $T_2$ and $T_3$. Furthermore, they gave a local criterion deducing the non-triviality of $\Pi_Z$ and the Ceresa cycles associated to $X_0(N)$. 

In this paper, we utilize a different perspective on this problem. We first pick a non-torsion Heegner divisor $P_{D,r}+P_{D,-r} -2H_D(\infty) \in \Jac(X)$ where $\infty$ is the cusp coming from $\infty \in \overline{\BH}$ and the construction of Heegner divisors is recalled in Section \ref{sec:Heegner}. The existence of such a divisor is guaranteed by the work of Gross--Zagier \cite{GZ} and Gross--Kohnen--Zagier \cite{GKZ} whenever there exists a weight $2$ normalized new form $f \in \mathcal{S}^{new}_2(\Gamma(N))^-$ satisfying $L'(f,1) \ne 0$. Then we show the existence of a linear combination of Hecke correspondences $Z \in \CH^1(X \times X)$ for which the Chow--Heegner divisor $\Pi_Z$ exactly equals $P_{D,r}+P_{D,-r}-2H_D\infty$. To show the existence of such a correspondence, we study the pullback of special divisors on $X \times X$ via the diagonal embedding $X \hookrightarrow X \times X$ and prove Theorem \ref{tthm:modularpullback}. 

Note that the existence of a non-torsion Heegner divisor $P_{D,r}+P_{D,-r}-2H_D\infty$ on a modular curve does not imply the nontriviality of its Ceresa or $\Delta_{GKS}$ cycle with respect to any base point, as can be seen for the modular curve $X_0(37)$. This curve is of genus $2$, hence hyperelliptic with trivial Ceresa cycle if a Weierstrass point (or the canonical class) is taken as base point. However, the space $\mathcal{S}^{new}_2(\Gamma_0(37))^-$ is of dimension $1$ with an element $f$ satisfying $L'(f,1) \ne 0$, which gives rise to a non-torsion Heegner divisor $P_{D,r}+P_{D,-r}-2H_D\infty$ on $X_0(37)$. This divisor cannot be realized as a Chow--Heegner divisor if the base point of the cycle $\Delta_{GKS}$ is a Weierstrass point, but it can be realized as a Chow--Heegner divisor if the base point is chosen to be the cusp $\infty$ on $X_0(37)$. This reflects the nontriviality of the Ceresa and $\Delta_{GKS}$ cycles associated with $X_0(37)$ when the base point is chosen to be a rational cusp.  We refer to Section \ref{subsec:basepoint} for more discussion. 

There are two reasons for us to work with the modular curve $X_N$ defined in Theorem \ref{thm:main1} instead of $X_0(N)$. For one, the canonical class of $X_N$ contains a divisor supported on cusps which we prove in Lemma \ref{l2}; the second reason is that we can explicitly construct special divisors on $X_N\times X_N$ which pull back to Heegner divisors on $X_N$ via the diagonal embedding, which we discuss in Section \ref{sec:explicitcover}. Actually, with a little more work, we can replace $X_N$ by $X_1(N)$.  But the pullback formula is not as pretty as Theorem \ref{tthm:modularpullback}, and we prefer to keep this exposition as clean as possible. In a follow-up paper, we will study pullback formulas between more general orthogonal-type Shimura varieties and deduce results on other modular curves and Shimura curves.  More generally, in Section \ref{S5.5}, we give a discussion on the scope of using the Chow--Heegner divisor construction to detect nontriviality of the Ceresa and $\Delta_{GKS}$ cycles for general algebraic curves.

Finally, Sections \ref{S5.3} and \ref{S5.4} address the effect of the choice of base point. We show that the difference of the Abel--Jacobi images of $\Delta_{GKS}$ cycles with two different base points lies in a specific direct summand of the intermediate Jacobian, called $J(\theta H^1)$.  On the other hand, when $X=X(\Gamma)$ with $\Gamma$ torsion-free and $\infty$ is a cusp, the Abel--Jacobi image of $\gks(X,\infty)$ lies in the complementary direct summand $J(\wpr)$.  This explains why, when the latter image is nontrivial, the change of base point cannot ``kill'' the nontriviality of the $\Delta_{GKS}$ cycle, and studying the Abel--Jacobi images allows us to deduce nontriviality of the Ceresa cycle from the nontriviality of the $\Delta_{GKS}$ cycle using work of Colombo--van Geemen \cite{CG}.

\vspace{.5 cm}

\textit{Acknowledgments:} The authors thank Jennifer Balakrishnan, Alexander Betts, Raymond van Bommel, Xiaojiang Cheng, Edgar Costa, Henri Darmon, William Duke, Jordan Ellenberg, Ziyang Gao, Sachi Hashimoto, Emmanuel Kowalski, Bjorn Poonen, Yunqing Tang, Boya Wen, Wei Zhang for helpful discussions and providing references throughout the work on this project.  

TY visited the Morningside Center of Mathematics in Beijing and Max Planck Institute for Mathematics at Bonn  while working on this project. He thanks both institutes for their  support and excellent working environment. 

MK, WL, and CQ attended the conference ``The Ceresa Cycle in Arithmetic and Geometry'' at ICERM while working on the project. They thank the organizers for their work and ICERM for its hospitality.

MK was partially supported by NSF grant DMS-2101482 and a gift from the Simons Foundation. WL was partially  supported by NSF grant DMS-2302511. TY was partially supported by UW-Madison's  Kellet mid-career award.

\section{The Ceresa cycle and the Gross--Kudla--Schoen cycle}\label{sec:cycle}

\subsection{The Ceresa cycle and the Gross--Kudla--Schoen cycle}

Let $X$ be a smooth projective curve defined over a field $K$ with a point $o \in X(K)$. Then associated to the data $(X,o)$, one can construct algebraic $1$-cycles on the triple product $X^3:=X \times X \times X$ and on the Jacobian of $X$ $\Jac(X)$, as we now describe.

For each non-empty set $\I \subset \{1,2,3\}$, let $\imath_{\I}^o: X \hookrightarrow X^3$ be the embedding sending a point $P \in X$ to $P$ for indices $i \in \I$ and to $o$ for indices $i \notin \I$. For each $\I \subset \{1,2,3\}$, we use $X_{\I}$ to denote the image $\imath^o_{\I}(X)$ in $X^3$.

\begin{defn}[\cite{GK,GS}]
    The \textit{Gross--Kudla--Schoen modified diagonal cycle}  associated to $(X,o)$ is the algebraic $1$-cycle in $X^3$ defined by
    \[ \Delta_{GKS}(X,o) := X_{123}-X_{12} -X_{23}-X_{13} + X_{1}+X_{2}+X_{3} \in Z^2(X^3). \]
The cycle $\Delta_{GKS}(X,o)$ gives rise to a class in $\CH^2(X^3)$ and is homologically trivial, also denoted by $\Delta_{GKS}(X,o)$.
\end{defn}

\begin{defn}[\cite{Ceresa}]
    The \textit{Ceresa cycle} associated to $(X,o)$ is the algebraic $1$-cycle in the Jacobian variety $\Jac(X)$ defined by 
    \[X^+_o - X^-_o \in Z^{g-1}(\Jac(X)) ,  \]
where $X^+_o$ is the embedding of $X$ in $\Jac(X)$ by $P\mapsto \AJ(P-o)=(P) -(o)$, and $X^-_o$ is the image of $X^+$ under the multiplication by $-1$ map on $\Jac(X)$.

    The cycle $X^+_o-X^-_o$ gives rise to a class in $\CH^{g-1}(\Jac(X))$ and it is homologically trivial.
\end{defn}

Although the cycles $\Delta_{GKS}(X,o)$ and $X^+_o-X^-_o$ live in different Chow groups, their cycle classes under the Abel--Jacobi maps are closely related. By that, we mean let $\mu:X^3 \to \Jac(X)$ be the map given by $(x,y,z) \mapsto x+y+z-3o$ and it induces a map between Chow groups mapping $\mu_*(\Delta_{GKS}(X,o)) \in \CH^{g-1}_0(\Jac(X))$. By the work of Colombo--van Geemen \cite[Proposition 2.9]{CG}, the Abel--Jacobi images of $X^+_o-X^-_o$ and $\mu_*(\Delta_{GKS}(X,o))$ differ by a constant multiple. We will use this fact in the proof of Theorem \ref{thm:main1} to deduce the nontriviality of $X^+_o-X^-_o$ from the nontriviality of $\gks(X,o)$. Moreover,
by \cite[Theorem 1.5.5]{Zhang}, the class $\Delta_{GKS}(X,\xi) \in \CH^2(X^3)$ is nontrivial if and only if the class $X^+_\xi-X^-_\xi\in \CH^{g-1}(\Jac(X))$ is nontrivial if the basepoint $\xi$ is a degree $1$ divisor on $X$ satisfying $(2g-2)\xi\equiv K_X \in \CH^1(X)$ modulo torsion. 

\subsection{The Chow--Heegner divisor}\label{sec:ChowHeegner} Here we want to introduce the Chow--Heegner divisor described in \cite{DRS} with a slightly different construction.

Let $Z \subset X\times X$ be a $1$-cycle; then associated to $Z$, one can obtain a degree $0$ divisor called \textit{the Chow--Heegner divisor} $\Pi_Z(\Delta_{GKS}(X,o)) \in \CH^1(X)$ defined below.

For each non-empty set $\I \subset \{1,2,3\}$, let $\pi_{\I}$ denote the projection map $$\pi_{\I}: X^3 \to \prod_{i \in \I} X_i.$$
We will use $\pi_{\I,*}$ and $\pi_{\I}^*$ to denote the pushforward and pullback maps on Chow groups induced by the map $\pi_{\I}$.

\begin{defn}[Chow--Heegner divisor \cite{DRS}]\label{def:ChowHeegner}
    Let $Z\in \CH^1( X\times X)$ be a correspondence between $X$ and itself.  We define a degree $0$ divisor $\Pi_Z(\Delta_{GKS}(X,o)) \in \CH^1(X)$ by
    \begin{equation}
        \Pi_Z(\Delta_{GKS}(X,o)):= \pi_{3,*}\left( (\pi_{12}^*Z)\cdot \Delta_{GKS}(X,o) \right)
    \end{equation}
    where ``$\cdot$'' denotes the intersection product in $\CH^{\bullet}(X^3)$.
\end{defn}

The degree-$0$ divisor $\Pi_Z(\Delta_{GKS}(X,o)) \in \CH^1(X)$ gives rise to a point in $\Jac(X)$. In the case where the curve $X$ is modular and $Z$ is a Hecke correspondence, this point $\Pi_Z(\Delta_{GKS}(X,o))\in\Jac(X)$ has been studied and referred to as a Chow--Heegner point in the literature including \cite{DF,DRS} because it is a linear combination of Heegner points

In the next lemma, we explicitly compute the divisor $\Pi_Z(\Delta_{GKS}(X,o))$.
For each non-empty set $\J \subset \{1,2\}$, let $\jmath^o_{\J}: X \hookrightarrow X^2$ be the embedding sending a point $P \in X$ to $P$ for indices $j \in \J$ and $P \mapsto o$ for indices $j \notin \J$.

\begin{lem}\label{lem222}
    Let $Z\in \CH^1( X\times X)$ be a correspondence between $X$ and itself. Then, we have
    \[ \Pi_Z(\Delta_{GKS}(X,o))= \jmath^*_{12} (Z)- \jmath_1^{o,*}(Z)-\jmath_2^{o,*}(Z)- \deg(\jmath^*_{12} (Z)- \jmath_1^{o,*}(Z)-\jmath_2^{o,*}(Z))o \]
in $\CH^1(X)$.
\end{lem}

\begin{proof}
    This lemma follows from the following direct computation.  First, in $\CH^3(X^3)$ we have 
    \begin{multline*}
\pi_{12}^*Z\cdot\left(X_{123}-X_{12}-X_{23}-X_{13}+X_1+X_2+X_3\right)=\\
(\imath_{123})_*\jmath_{12}^*Z-(\imath_{12}^o)_*\jmath_{12}^*Z-(\imath_{23}^o)_*(\jmath_2^o)^*Z-(\imath_{13}^o)_*(\jmath_1^o)^*Z
+(\imath_1^o)_*(\jmath_1^o)^*Z+(\imath_2^o)_*(\jmath_2^o)^*Z+0.
\end{multline*}
The last term is zero because $\pi_{12}^*\colon \CH^*(X^2)\to \CH^*(X^3)$ is a ring homomorphism, and so $$\pi_{12}^*Z\cdot X_3=\pi_{12}^*Z\cdot\pi_{12}^*\{(o,o)\}=\pi_{12}^*(Z\cdot\{(o,o)\})=0$$ as $Z\in \CH^1(X^2)$ and $\{(o,o)\}\in \CH^2(X^2)$ have product in $\CH^3(X^2)=\{0\}$.  Applying $(\pi_3)_*$ and reordering terms now gives
$$\jmath_{12}^*Z-(\jmath_1^o)^*Z-(\jmath_2^o)^*Z-(\deg(\jmath_{12}^*Z)-\deg((\jmath_1^o)^*Z)-\deg((\jmath_2^o)^*Z))o,$$
hence the result.
\end{proof}

From the construction, it is clear that if $\Delta_{GKS}(X,o) \in \CH^2(X^3)$ is a trivial class then the class $\Pi_Z(\Delta_{GKS}(X,o)) \in \Pic^0(X)$ is torsion for any $Z \in \CH^1(X \times X)$.  We state this in contrapositive form in the following Lemma.

\begin{lem}\label{lem:Z}
    If there exists a correspondence $Z \in \CH^1(X \times X)$ such that $\Pi_Z(\Delta_{GKS}(X,o))\in \Jac(X)$ is non-torsion, then $\Delta_{GKS}(X,o) \in \CH^2(X^3)$ is nontrivial. 
    
    Moreover, if the base point $o$ is a degree $1$ divisor on $X$ satisfying $(2g-2)o=K_X \in \CH^1(X)$,  following \cite[Theorem 1.5.5]{Zhang}, the Ceresa cycle $X^+_o-X^-_o\in\CH^{g-1}(\Jac(X))$ is also nontrivial.
\end{lem}

This principle is applied repeatedly in sections \ref{sec:Heegner}-\ref{S5.4}. It is a natural and interesting problem to determine when the non-triviality of the $\Delta_{GKS}(X,o)$ cycle can be detected by $\Pi_Z(\Delta_{GKS}(X,o))$ for some correspondence $Z \in \CH^1(X\times X)$, and we give some Hodge theoretic discussion on this matter in Section \ref{S5.5}.

\subsection{Examples of Chow--Heegner divisors}\label{subsec:basepoint}
In this section, we give two examples of Chow--Heegner divisors which can be computed using Lemma \ref{lem222}. These examples have been computed in \cite{DRS} and \cite{DDLR}, we put them here to help our discussion.

\begin{eg}[\cite{DRS}*{Corollary 2.8}]\label{eg:diagonal}
    Let $Z \subset X \times X$ be the diagonal, then the Chow--Heegner divisor 
    \[ \Pi_Z(\Delta_{GKS}(X,o)) = (2g-2)o - K_X  \]
    which is computed in Equation \ref{en7}.
    
    In particular, if the base point $o$ does not satisfy $(2g-2)o=K_X \in \CH^1(X)$, then the cycle $\Delta_{GKS}(X,o) \in \CH^2(X^3)$ is nontrivial. If the goal is to study the triviality/nontriviality of the cycle $\Delta_{GKS}(X,o)$, then the only interesting choice is to take $o$ to be a degree $1$ divisor in $\CH^1(X)$ satisfying $(2g-2)o=K_X \in \CH^1(X)$ as in the work of Zhang \cite{Zhang}.
\end{eg}

\begin{eg}[\cite{DDLR}*{Section 5.1}]\label{eg:X0(37)}
    Let $X$ be the modular curve $X_0(37)$ and $o=\infty$ be the rational cusp coming from $\infty \in \overline{\BH}$. Then $X$ has automorphism group isomorphic to $\BZ/2\BZ \times \BZ/2\BZ$ generated by the Atkin--Lehner involution $\omega$ and the hyperelliptic involution $S$. (See \cite{MSD} for a detailed analysis of this curve.) Let $T$ be the involution given by $T=S \circ \omega=\omega\circ S$ and $Z \subset X\times X$ is the graph of $T$, meaning points on $Z$ are of the form $(P,T(P)) \in X \times X$. Then the Chow--Heegner divisor $\Pi_Z(\Delta_{GKS}(X,o))$ is a rational multiple of any rational Heegner divisor $P_{D,r}+P_{D,-r}-2H_D\infty$ on $X$. In particular, the Chow--Heegner divisor $\Pi_Z(\Delta_{GKS}(X,o))$ is non-torsion in $\Jac(X)$.

    On the other hand, if we take $o \in X(\overline{\BQ})$ to be a Weierstrass point, then the Chow--Heegner divisor $\Pi_Z(\Delta_{GKS}(X,o))$ is torsion in $\Jac(X)$. (Moreover, we have $(2g-2)o=K_X$, which is always true for Weierstrass points.) This agrees with the statement \cite[Proposition 4.8]{GS} that the cycle $\Delta_{GKS}(X,o)$ is trivial in $\CH^2(X^3)$ when $X$ is a hyperelliptic curve and $o$ is a Weierstrass point.
\end{eg}

The phenomenon in Example \ref{eg:X0(37)} occurs because the canonical class of $X$ does not contain a divisor supported on cusps.  As we show in Lemma \ref{l2}, for $X=X(\Gamma)$ with $\Gamma$ torsion free (e.g.~$X=X_N$), $K_X$ always \emph{does} contain such a divisor.  

\section{Pull-back of special divisors on modular curves}\label{sec:explicitcover}

In this section, we will study the pullback of special divisors on a modular curve $X \hookrightarrow X \times X$ via the diagonal embedding. Although we don't prove all special divisors on $X_0(N)$ can be obtained from the pullback of divisors on $X_0(N) \times X_0(N)$, we will construct a cover $\pi_W:X_N \to X_0(N)$ such that the preimage of Heegner divisors on $X_0(N)$ under $\pi_W$ can be obtained from pullback via the diagonal map $X_N \hookrightarrow X_N \times X_N$.

\subsection{Modular curves and product of modular curves as orthogonal Shimura varieties}\label{subsec:modularcurve} We start by recalling the definition of orthogonal Shimura varieties and restrict to the case of a modular curve $X$ and the product $X \times X$. We refer to \cite[Section 7.1]{BY} and \cite{YY} for this material.

Let $(W,Q)$ be a quadratic space over $\mathbb{Q}$ of signature $(n,2)$ with its bilinear form denoted as $(\cdot,\cdot)$. The Hermitian symmetric domain corresponding to $\GSpin(W)$ can be realized as a space of lines
\[\BD_W = \{  z \subset W(\BC) \mid Q(z)=0, (z,\bar{z})<0\}/\BC^\times.\]
For a compact open subgroup $K \subset \GSpin(W)(\BA_\rf)$, we obtain a Shimura variety
\begin{equation}\label{eq:complexuni}
    \Sh(W)_K = \GSpin(W)(\BQ)\backslash (\BD_W\times \GSpin(W)(\BA_\rf)/K).
\end{equation}
The Shimura variety $\Sh(W)_K$ is a quasi-projective variety of dimension $n$ defined over $\BQ$.

Every element of $\BD_W$ gives a $\BC$-line over a point on $\Sh(W)_K$. This line bundle is called \textit{the tautological line bundle} and we will denote it by $\mathcal{L}_W$.

Let $N$ be a positive integer. Define the quadratic space 
\[W:= \{ x \in M_2(\BQ) \mid \tr(x)=0 \}, \]
with the quadratic form $Q(x)=N\det(x)$, and the corresponding bilinear form $$(x,y)=N\cdot \tr(x\cdot  \textup{adj}(y)).$$ 
The space $W$ has signature $(1,2)$ and the group $\GSpin(W) \cong \GL_2$ acts on $W$ by conjugation.  The space $\BD_W$ can be identified with the union of the upper and lower half planes $\BH \cup \BH^-$ by
\begin{align*}
    \BH \cup \BH^- & \to \BD_W, \\
    z &\mapsto \begin{pmatrix}
        z & -z^2 \\ 1 & -z
    \end{pmatrix} .
\end{align*}
Under this identification, the conjugation action of $\GL_2$ on $\BD_W$ corresponds to linear fractional transformations on $\BH \cup \BH^-$.

Note that the space $\BD_W$ and the group $\GSpin(W)$ do not depend on the integer $N$. The special divisors to be defined in next subsection will depend on $N$.

Consider the compact open subgroup $K_0(N)\subset \GL_2(\BA_\rf)$ defined as
\[K_0(N) = \biggl\{\begin{pmatrix} a & b \\ c & d \end{pmatrix} \in \hbox{GL}_2(\hat \BZ):\,  c \equiv 0 \pmod N \biggr\}.\]

Since $K_0(N) \cap \GL_2^+(\BQ) = \Gamma_0(N)$, the corresponding Shimura variety $\Sh(W)_{K_0(N)}$ is nothing but the open modular curve $Y_0(N) = \Gamma_0(N) \backslash \BH$. Similarly, if we consider groups 
\begin{align*}
K_1(N) &= \biggl\{\begin{pmatrix} a & b \\ c & d \end{pmatrix} \in \hbox{GL}_2(\hat \BZ):\,  c \equiv 0 \bmod N,  d \equiv 1  \bmod N \biggr\},
\\
K(N) &= \biggl\{\begin{pmatrix} a & b \\ c & d \end{pmatrix} \in \hbox{GL}_2(\hat \BZ):\,  b \equiv  c \equiv 0 \bmod N,  d \equiv 1  \bmod N \biggr\},
\end{align*}
we obtain modular curves $Y_1(N) = \Sh(W)_{K_1(N)}$ and $Y(N)=\Sh(W)_{K(N)}$. For any congruence subgroup $\Gamma(N) \subset \Gamma \subset \SL_2(\BZ)$, we can take the group $K(\Gamma)$ to be the product of 
\begin{itemize}[leftmargin=1cm]
\item $\nu(\hat{\BZ}^\times)$ (where $\nu:\BA_\rf^\times \to \GL_2(\BA_\rf)$ is given by $\nu(a)={\rm diag}(1,a)$), and
\end{itemize}
\begin{itemize}[leftmargin=1cm]
\item the preimage of $\Gamma/\Gamma(N)$ (under the natural map $\GL_2(\hat{\BZ}) \to \GL_2(\BZ/N\BZ)$) in $\GL_2(\hat{\BZ})$,
\end{itemize}
then we obtain $\Sh(W)_{K(\Gamma)} = Y_\Gamma=\Gamma\backslash\BH$.

The modular curve $Y_\Gamma$ is open and we can compactify it by $X_\Gamma=\Gamma\backslash\overline{\BH}$, where $\overline{\BH}=\BH \cup \BP^1(\BQ)$. The points $\Gamma\backslash\BP^1(\BQ)$ on $X_\Gamma$ are called \textit{cusps}. The tautological line bundle $\mathcal{L}_W$ extends to $X_\Gamma$; its square $\mathcal{L}_W^2$ is the line bundle of modular forms of weight $2$ on $X_\Gamma$, which identifies with the canonical line bundle $\omega_W$ associated to the canonical divisor class $K_{X_\Gamma}$.

Next we consider the product $Y_\Gamma \times Y_\Gamma$ as an orthogonal Shimura variety where $Y_\Gamma=\Sh(W)_{K(\Gamma)}$ is a modular curve.

Define a quadratic space 
\[V:= \{x \in M_2(\BQ)  \}\]
with quadratic form $Q(x)=N \det (x)$. Then $V$ has signature $(2,2)$ and $W$ is naturally a quadratic subspace of $V$. The group $\GSpin(V)$ can be explicitly written as 
\[\GSpin(V) = \{ (b_1,b_2): b_1,b_2 \in \GL_2(\BQ), \det b_1 =\det b_2 \}, \]
and $(b_1,b_2)$ acts on $x \in V$ as $x \mapsto b_1xb_2^{-1}$.
The space $\BD_V$ can be identified with $\BH^2\cup (\BH^-)^2$ by
\begin{align*}
    \eta:\BH^2 \cup (\BH^-)^2 & \to \BD_V, \\
    (z_1,z_2) &\mapsto \begin{pmatrix}
       z_1  & -z_1z_2 \\ 1 & -z_2
    \end{pmatrix} .
\end{align*}
Under this identification, for an element $(b_1,b_2) \in \GSpin(V)$, its action on $\BD_V$ corresponds to the usual linear fractional transformation $(b_1,b_2)(z_1,z_2) = (b_1z_1,b_2z_2)$ on $\BH^2\cup(\BH^-)^2$.

Considering the compact open subgroup $$\tilde{K}_0:= (K_0(N)\times K_0(N))\cap \GSpin(V)(\BA_\rf) \subset \GSpin(V)(\BA_\rf),$$ we get the corresponding Shimura variety $$\Sh(V)_{\tilde{K}_0} = Y_0(N) \times Y_0(N) = \Gamma_0(N) \times \Gamma_0(N) \backslash \BH^2.$$ Similarly, for any congruence subgroup $\Gamma\subset \SL_2(\BZ)$,  by setting $\tilde{K}(\Gamma)=(K(\Gamma)\times K(\Gamma))\cap \GSpin(V)(\BA_\rf)$, we obtain $\Sh(V)_{\tilde{K}(\Gamma)} = Y_\Gamma \times Y_\Gamma = \Gamma \times \Gamma \backslash \BH^2$. We will denote the group $\Gamma \times \Gamma$ by $\Gamma_V$.

Recall that the space $W$ is naturally a quadratic subspace of $V$ which induces a natural embedding $Y_\Gamma \hookrightarrow Y_\Gamma \times Y_\Gamma$. This embedding is nothing but the diagonal embedding.

We can compactify $Y_\Gamma\times Y_\Gamma$ by considering $X_\Gamma \times X_\Gamma = \Gamma_V \backslash\overline{\BH}^2$. Like the case for $Y_\Gamma\hookrightarrow X_\Gamma$, the tautological line bundle $\mathcal{L}_V$ extends to $X_\Gamma \times X_\Gamma$, with sections given by modular forms in $2$ variables of weight $(1,1)$. Its square $\mathcal{L}_V^2$ is the line bundle of two-variable holomorphic modular forms of weight $(2,2)$, which identifies with the canonical line bundle $\omega_V$ on $X_\Gamma\times X_\Gamma$.

\subsection{Special divisors on modular curves and product of modular curves}\label{subsec:specialdivisor}
An orthogonal Shimura variety $\Sh(W)_K$ naturally comes with families of algebraic cycles arising from quadratic subspaces of $W$. In this section, we discuss special divisors on a modular curve $Y_\Gamma$ and its self-product $Y_\Gamma \times Y_\Gamma$, coming from certain Schwartz functions on $W(\BA_\rf)$ and $V(\BA_\rf)$ respectively. We start by constructing some lattices in $W$ and $V$.

Let $L_W \subset W$ be the following lattice 
\[ L_W:= \biggl\{ \begin{pmatrix}
    a & -b/N \\ c & -a
\end{pmatrix} \mid a,b,c \in \BZ \biggr\} \]
on which $Q(x)=N\det(x)$ takes integral values.

Then its dual lattice with respect to the bilinear form $(x , y ) = N \cdot \tr(x\cdot {\rm adj}(y))$ is
\[  L_W':= \biggl\{ \begin{pmatrix}
    a/2N & -b/N \\ c & -a/2N
\end{pmatrix} \mid a,b,c \in \BZ \biggr\}. \]

The group $L_W'/L_W$ is isomorphic to $\BZ/2N\BZ$ with isomorphism given by 
\begin{equation}\label{mur}
 \BZ/2N\BZ \to L'_W/L_W: \quad   r \mapsto \mu_r= \begin{pmatrix}
    r/2N & 0 \\ 0 & -r/2N
\end{pmatrix} .
\end{equation}

For any $\mu \in L_W'/L_W$, we obtain a Schwartz function on $W(\BA_\rf)$ given by 
\[ \phi_\mu = \Char(\mu +\hat{L}_W),\quad \text{where } \hat{L}_W=L_W \otimes \hat{\BZ}. \]

Note that the group $K_0(N)$ action preserves $\hat{L}_W$ and $\hat{L}'_W$. Moreover, the group $K_0(N)$ acts trivially on $L'_W/L_W \simeq \hat{L}'_W/\hat{L}_W$ by conjugation. Thus, the same holds for any subgroup $K(\Gamma) \subset K_0(N)$ which corresponds to a congruence subgroup $\Gamma\subset \SL_2(\BZ)$. Let $\Gamma$ be such a group and let $Y_\Gamma = \Gamma\backslash\BH$ be the corresponding modular curve, as constructed in Section \ref{subsec:modularcurve}.

Given an element $\mu_0 \in L'_W/L_W$ and a positive number $m_0 \in \BQ_{>0}$, we can define a special divisor
$Z_{Y_\Gamma}(m_0,\mu_0)$ on the modular curve $Y_\Gamma$ as 
\[Z_{Y_\Gamma}(m_0,\mu_0) = \sum_{x\in \Gamma\backslash \Omega_{m_0}(\BQ)} \phi_{\mu_0}(x)\textup{pr}(\BD_{x,W},1)\]
where $\Omega_{m_0} = \{x\in W \mid Q(x)=m_0\}$,  $\textup{pr}$ is the natural projection from $\BD_W \times \GSpin(W)(\BA_\rf)$ to $\Gamma\backslash\BH$ coming from Equation \eqref{eq:complexuni},  $\BD_{x,W}=\{z\in \BD_W \mid z \perp x\}$, and $\phi_{\mu_0}=\textup{char} (\mu_0 + \hat{L}_W)$.

If we let $\pi_\Gamma: Y_\Gamma \to Y_0(N)$ denote the natural map given by $\Gamma\backslash\BH \to \Gamma_0(N)\backslash\BH$, then by definition
\begin{equation}\label{eq:Zm0mu0}
    Z_{Y_\Gamma}(m_0,\mu_0) = \pi_\Gamma^*Z_{Y_0(N)}(m_0,\mu_0).
\end{equation}

\begin{rmk}\label{rmk:Heegner}
    We can follow \cite[Section 7.1]{BY} to describe elements $\mu_0 \in L_W'/L_W$ and $m_0\in \BQ_{>0}$ such that
\[Z_{Y_0(N)}(m_0,\mu_0) = P_{D,r}+P_{D,-r},\]
where $D=-4Nm_0\in \BZ$ is a negative discriminant, $r \in \BZ/2N\BZ$ satisfies $r^2 \equiv D \bmod 4N$, and $P_{D,r}+P_{D,-r}$ is a Heegner divisor on $Y_0(N)$ as described in \cite[Section IV]{GKZ} (and recalled in Section \ref{sec:Heegner} below).  In this case, we just need to take $\mu_0 = \mu_r$ following Equation \eqref{mur}. We will use this construction in Section \ref{sec:mainproof}.
\end{rmk}

For any special divisor $Z_{Y_\Gamma}(m_0,\mu_0)$, we denote its Zariski closure in $X_\Gamma$ by $Z_{X_\Gamma}(m_0,\mu_0)$. Such a divisor is supported on CM points, and $Z_{X_\Gamma}(m_0,\mu_0)$ is simply the image of $Z_{Y_\Gamma}(m_0,\mu_0)$ under the natural embedding $Y_\Gamma \hookrightarrow X_\Gamma$.

Now we extend the definition of $Z_{X_\Gamma}(m_0,\mu_0)$ to the case where $m_0=0$.  For any $0\ne x \in W$ with $Q(x)=0$,  there does not exist any negative $2$-plane $z \in \BD_W$ such that $x \perp z$. So we can naturally define $Z_{X_\Gamma}(0, \mu_0)=Z_{Y_\Gamma}(0, \mu_0) =0$ for $\mu_0 \ne 0$.  On the other hand,  every point in $\mathbb D_W$ is perpendicular to $0$, so formally $Z_{Y_\Gamma}(0, 0) =Y_\Gamma$, and $Z_{X_\Gamma}(0, 0)=X_\Gamma$ which has of course the wrong dimension. By the adjunction formula, it is natural to define $Z_{X_\Gamma}(0, 0) = [\omega_W^{-1}]=-K_{X_\Gamma} \in \hbox{CH}^1(X_1)$ to be the negative of the canonical divisor (associated to the canonical line bundle on $X_1$).

Next we want to construct some special divisors on $Y_\Gamma \times Y_\Gamma$ whose pullbacks under the diagonal embedding $Y_\Gamma \hookrightarrow Y_\Gamma \times Y_\Gamma$ yield the special divisors $Z_{Y_\Gamma}(m_0,\mu_0)$. We start by constructing a lattice in $V$ which is closely related to the lattice $L_W$. 

Recall that the space $V$ decomposes as $V = W \oplus W^{\perp}$, where 
\[W^\perp = \biggl\{ \begin{pmatrix}
    a & 0 \\ 0 & a
\end{pmatrix} \mid a \in \BQ \biggr\}. \]
Define $P\subset W^\perp$ by
$$
P= \{ a I_2 \mid  a \in \mathbb Z \} \cong  \mathbb Z, 
$$
with quadratic form $Q(a) =N a^2$.  Its dual lattice with respect to $(\cdot,\cdot)$ is $P' = \{(a/2N)I_2 \mid a \in \BZ \}$ and $P'/P \simeq \BZ/2N\BZ$. 

Now consider the lattice $L= L_W  \oplus P \subset V$. Explicitly, 
\[  L:= \biggl\{ \begin{pmatrix}
    a & -b/N \\ c & d
\end{pmatrix} \mid a,b,c,d \in \BZ , a \equiv d  \bmod 2 \biggr\},  \]
and its dual lattice with respect to $(x,y) = N \cdot \tr(x \cdot {\rm adj}(y))$ is
\[  L':= \biggl\{ \begin{pmatrix}
    a/2N & -b/N \\ c & d/2N
\end{pmatrix} \mid a,b,c,d \in \BZ , a \equiv d  \bmod 2 \biggr\}.  \]
Elements in $L'/L$ can be represented by $\begin{pmatrix}
    r_1+r_2 & 0 \\ 0 & r_1-r_2
\end{pmatrix}$ where $r_1,r_2 \in \{0,\cdots, \frac{2N-1}{2N}\} \subset \mathbb{Z}$. 
Note that $L'/L = L_W'/L_W \oplus P'/P$ as $\begin{pmatrix}
    r_1+r_2 & 0 \\ 0 & r_1-r_2
\end{pmatrix} = \begin{pmatrix}
    r_1 & 0 \\ 0 & r_1
\end{pmatrix} + \begin{pmatrix}
    r_2 & 0 \\ 0 & -r_2
\end{pmatrix}$ and this decomposition is unique.

For any $\mu \in L'/L$, we obtain a Schwartz function on $V(\BA_\rf)$ given by 
\[\phi_\mu = \Char(\mu+\hat{L}), \quad\text{where } \hat{L}=L \otimes \hat{\BZ}.\]

\begin{rmk}
We would like to construct special divisors of the form $Z(m,\mu)$ on $Y_0(N) \times Y_0(N)$ using the lattice $L$ and a Schwartz function $\phi_\mu$. But unfortunately, the group $K_0(N) \times K_0(N)$ does not act on $L'/L$ trivially (moreover, $K_0(N) \times K_0(N)$ does not fix $L$), and so the special divisors $Z(m, \mu)$ will not be well-defined in $Y_0(N) \times Y_0(N)$.
This is the reason why we pass to the cover $Y_N \to Y_0(N)$ as described below, which will allow us to obtain special divisors $Z_{Y_N}(m_0,\mu_0)$ via pullback by the map $Y_N \hookrightarrow Y_N \times Y_N$. 
\end{rmk}


Define groups 
$K=K_1(2N) \cap K(2)$ and $K_V = (K \times  K)\cap \GSpin(V)(\mathbb A_f)$.
Then, explicitly we can write out
\[K = \biggl\{ \begin{pmatrix}
    a & 2b \\ 2Nc & 1+2Nd
\end{pmatrix} \in \GL_2(\hat{\BZ}) \mid a,b,c,d \in \hat{\BZ} \biggr\} .  \]
By direct computation, one can check that \emph{any element $ (b_1,b_2) \in K_V$ acts on $x\in L'/L$ trivially under the action $b_1xb_2^{-1}$.} 

Define the group $\Gamma_N = K\cap \hbox{GL}_2^+(\mathbb Q)=\Gamma_1(2N) \cap \Gamma(2)$
 and let $Y_N =\Gamma_N \backslash \mathbb H $  be the associated modular curve. Let $\Gamma_V = K_V \cap \GSpin(V_\BQ)^+$, where $+$ means that the spinor norm is positive. Then $\Gamma_V=\Gamma_N \times \Gamma_N$, and the Shimura variety associated to $K_V$ is $Y_N \times Y_N = \Gamma_V\backslash\BH^2$. 

 For any $m \in \BQ_{> 0}$, define set $\Omega_m = \{x\in V \mid Q(x)=m\}$. Then by \cite[Lemma 4.1]{BY}, we can define a special divisor of $Y_N \times Y_N$ by 
\[Z(m,\mu) = \sum_{x\in \Gamma_V \backslash \Omega_m(\BQ)} \phi_\mu(x)\textup{pr}(\BD_x,1)\]
where $\textup{pr}$ is the natural projection from $\BD_V \times \GSpin(V)(\BA_\rf)$ to $\Gamma_V\backslash\BH^2$ from Equation \eqref{eq:complexuni}, $\BD_x=\{z\in \BD_V \mid z \perp x\}$, and $\phi_\mu=\textup{char} (\mu + L\otimes\hat{\BZ})$ for any $\mu \in L'/L$. For any $Z(m,\mu)$ in $Y_N\times Y_N$, let $Z^*(m, \mu)$ be the Zariski closure of $Z(m,\mu)$ in $X_N \times X_N$.

For our purposes, we want to extend the definition of $Z^*(m,\mu)$ to the case where $m=0$.  For any $x \in V$, if $x \ne 0$ and $Q(x)=0$, then there does not exist any negative $2$-plane $z \in \BD_V$ such that $x \perp z$. So, for any $\mu \ne 0$, we can simply extend our definition and make $Z^*(0,\mu)=0$. In the case where $\mu=0$, we define $Z^*(0,0) = [\omega_V^{-1}]$.

In the following lemma, we show that the difference $Z^*(m, \mu) - Z(m, \mu)$ is supported on points $(P_1,P_2) \in X_N\times X_N$ where both $P_1$ and $P_2$ are cusps.

\begin{lem}\label{lem:boundary} Let the notation be as above. Then 
$$
Z^*(m, \mu) - Z(m, \mu) = \sum_{ \begin{matrix} \gamma  \in  \mu +L , Q(\gamma) =m
\\
 \mod  \Gamma_V 
\end{matrix} }
   \Gamma_{V, \gamma} \backslash  \{ (\gamma P,  P):\, P \in  \mathbb P^1(\mathbb Q)\}.
$$
Here $\Gamma_{V, \gamma}$  is the stabilizer of $\gamma$ in $\Gamma_V$ and $\gamma \in \mu+L \subset M_2(\BQ)$ acts on $\BP^1(\BQ)$ by fractional linear transformations.
\end{lem}
\begin{proof} 
Recall the map $\eta: \BH^2 \to \BD_V$ given by $\eta(z_1,z_2) = \begin{pmatrix}
    z_1 & -z_1z_2 \\ 1 & -z_2
\end{pmatrix}$.

By direct calculation, we get 
$$
Z(m, \mu) = \sum_{ \begin{matrix} \gamma  \in  \mu +L , Q(\gamma) =m
\\
 \mod  \Gamma_V 
\end{matrix} }  Y_\gamma
$$
where    
$$
Y_\gamma =\Gamma_{V, \gamma} \backslash \{ (z_1, z_2) \in \mathbb H^2:\,  \gamma \perp \eta(z_1, z_2) \}.
$$
Notice that  $ \gamma \perp \eta(z_1, z_2)$ if and only if $z_1 = \gamma z_2$, where $\gamma$ acts on $\BH$ by fractional linear transformation.  So 
$$
Y_\gamma = \Gamma_{V, \gamma} \backslash  \{ (\gamma z, z):\,  z \in \mathbb  H\}
$$
is a modular curve. Its closure in $X_1 \times X_1$
is thus 
$$
Y_\gamma \cup \Gamma_{V, \gamma} \backslash  \{ (\gamma P,  P):\, P \in  \mathbb P^1(\mathbb Q)\}.
$$
This proves the lemma.
\end{proof}

\subsection{Pullback of special divisors by the diagonal embedding}
In Sections \ref{subsec:modularcurve} and \ref{subsec:specialdivisor}, we constructed modular curves $X_0(N),X_N$ and their products $X_0(N)\times X_0(N)$, $X_N\times X_N$.  They form the following commutative diagram:
\begin{center}
    \begin{tikzcd}
       X_N 
\arrow[d,"\pi_W"] \arrow[r,"\iota_1"] &   X_N \times X_N \arrow[d,"\pi_V"] \\
 X_0(N) \arrow[r,"\iota"] & X_0(N) \times X_0(N)
    \end{tikzcd}
\end{center}
where $\iota_1$ and $\iota$ are diagonal embeddings, and the maps $\pi_W$ and $\pi_V$ are natural projections.  The same commutative diagram holds for the open Shimura varieties with $X$ replaced by $Y$. Note that $X_N$ is connected by the definition of the group $\Gamma_N$.

In Section \ref{subsec:specialdivisor}, we have constructed special divisors $Z_{X_N}(m_0,\mu_0) \in \CH^1(X_N)$ and $Z^*(m,\mu) \in \CH^1(X_N\times X_N)$. For a special divisor $Z^*(m,\mu)\in \CH^1(X_N\times X_N)$, we will compute its image under the pullback map $\iota^*_1$ and then show any divisor of the form $Z_{X_N}(m_0,\mu_0) \in \CH^1(X_N)$ is the image of this pullback map $\iota^*_1$ for some linear combination of $Z^*(m,\mu)$'s.

As was mentioned in \cite[Corollary 3.3]{YY}, the diagonal $X_N \subset X_\Gamma\times X_\Gamma$ can be viewed as a special divisor of $X_\Gamma\times X_\Gamma$. Thus, to discuss the pullback by $\iota^*$, we start with this diagonal.

The adjunction formula asserts 
\begin{equation} \label{eq: Adj}
    \iota_1^* X_N = [\omega_W ]-  [\iota_1^* \omega_V]  = [ \omega_W^{-1}]  =Z_{X_N}(0, 0)\in  \hbox{CH}^1(X_N).
\end{equation}
Here we use the fact that $\iota_1^*(\omega_V)$ is the line bundle of modular forms of weight $4$ and is thus isomorphic to $\omega_W^2$.

Now we can compute an explicit formula for the pullback $\iota_1^*: \CH^1(X_N\times X_N) \to \CH^1(X_N)$ in the same fashion as \cite[Equation (8.6)]{BY} .

\begin{lem}\label{lem:pullback}
    With notations as defined above, fix $m>0$, and $\mu \in L'/L$.
    Recall $L'/L=L_W'/L_W \oplus P'/P$, and there is a unique decomposition $\mu=\mu_0+\mu^+$ where $\mu_0 \in L_W'/L_W$ and $\mu^+\in P'/P$.
    We have 
\begin{equation}\label{eq:pullback}
    \iota_1^* Z^*(m, \mu) 
    = \sum_{ \begin{matrix}
        m =m_0 + m^+,\\ m_0,m^+ \ge 0, \\ m_0 \equiv Q(\mu_0) \pmod 1,\\
        m^+ \equiv Q( \mu^+) \pmod 1\\  
    \end{matrix}} Z_{X_N}(m_0, \mu_0) a_P(m^+, \mu^+)  + c (\infty) \in \CH^1_{\mathbb Q}(X_N)
\end{equation}
for some integer $c \in \mathbb Z$
where 
$$
a_P(m^+,\mu^+) = \# \{ x^+ \in \mu^+ + P |\,  Q(x^+)=m^+ \}.$$  
Here  $(\infty)$ is  the cusp of $X_N$ at $\infty$.
\end{lem} 
\begin{proof} This identity holds in general on the open Shimura varieties at least for $a_P(m, \mu) =0$.  When $a_P(m, \mu) = 0 $, $\iota_1(X_N)$ intersects $Z(m,\mu)$ properly, and direct calculation gives 
$$
 \iota_1^* Z(m, \mu) 
    = \sum_{ \begin{matrix}
        m =m_0 + m^+,\\ m_0,m^+ \ge 0, \\ m_0 \equiv Q(\mu_0) \pmod 1,\\
        m^+ \equiv Q( \mu^+) \pmod 1\\  
    \end{matrix}} Z_{Y_N}(m_0, \mu_0) a_P(m^+, \mu^+) 
$$
as divisors.  As $Z^*(m, \mu)$ might intersect with boundary of $X_N \times X_N$, its pullback might include cusps following Lemma \ref{lem:boundary},  and thus Equation \eqref{eq:pullback} holds with possibly some extra $c (\infty)$. Here we use the well-known result of Drinfeld \cite{Drinfeld} and Manin \cite{Manin} that the difference of two cusps is a torsion divisor.

When  $a_P(m, \mu)\ne 0 $ (which can only happen in the case $\mu \in P'/P$), $\iota_1 (X_N)$ appears in $Z^*(m, \mu) $  with multiplicity $a_P(m, \mu)$. The same calculation and Equation \eqref{eq: Adj} gives 
\begin{align*}
  \iota_1^* Z^*(m, \mu) 
   & = \sum_{ \begin{matrix}
        m =m_0 + m^+,\\ m_0 > 0 ,m^+ \ge 0, \\ m_0 \equiv Q(\mu_0) \pmod 1,\\
        m^+ \equiv Q( \mu^+) \pmod 1\\  
    \end{matrix}} Z_{X_N}(m_0, \mu_0) a_P(m^+, \mu^+)  + c (\infty)
     + a_P(m, \mu) \iota_1^* X_N
     \\
    &= \sum_{ \begin{matrix}
        m =m_0 + m^+,\\ m_0  ,m^+ \ge 0, \\ m_0 \equiv Q(\mu_0) \pmod 1,\\
        m^+ \equiv Q( \mu^+) \pmod 1\\  
    \end{matrix}} Z_{X_N}(m_0, \mu_0) a_P(m^+, \mu^+)  + c (\infty)
\end{align*}
as claimed.
\end{proof}

\begin{prop}\label{prop:mainprop}
    Any special divisor of the form $ Z_{X_N}(m_0,\mu_0) \in\CH_{\mathbb Q}^1(X_N)$ with $\mu_0 \in L_W'/L_W$ is the pull-back of a rational linear combination of  special divisors on $X_N \times X_N$ by the map $\iota_1$.  Here we view $Z^*(0, 0)=[\omega_V^{-1}] $ as a special divisor on $X_N \times X_N$.
\end{prop}

\begin{proof} The proof is by induction on $m_0$. The case $m_0=0$ is clear following Equation \eqref{eq: Adj}.   Note that $(\infty) = a [\omega_W] $ for some rational number $a$ for any modular curve $X_{\Gamma}$ with $\Gamma$ torsion-free (see the proof of Lemma \ref{l2}).  In particular, this holds for $X_N$.

For $m_0 >0$, by Lemme \ref{lem:pullback}, we have 
$$
Z_{X_N}(m_0, \mu_0) 
 = \iota_1^* Z^*(m_0, \mu_0)  - \sum_{\begin{matrix}
        m_0 =m_0' + m^+,\\m_0> m_0' \ge 0 \end{matrix} } Z_{X_N}(m_0', \mu_0) a_P(m^+, 0)  - c (\infty).
$$
Now the proposition is clear by induction.
\end{proof} 

Note that Theorem \ref{tthm:modularpullback} stated in the introduction directly follows from Proposition \ref{prop:mainprop}, Remark \ref{rmk:Heegner}, and Equation \eqref{eq:Zm0mu0}.

\section{Non-triviality of the Ceresa cycle and the Gross--Kudla--Schoen cycle associated to modular curves}\label{sec:mainproof}

In this section, we use the Chow--Heegner divisor construction described in Section \ref{sec:ChowHeegner} and the explicit cover and pullback formula in Proposition \ref{prop:mainprop} to deduce our main results on the nontriviality of the Ceresa and Gross--Kudla--Schoen cycles associated to the modular curves $X_N$ with respect to any base point.

\subsection{Heegner divisor and Chow--Heegner divisor}\label{sec:Heegner}
In this section, we use the Chow--Heegner divisor discussed in Section \ref{sec:ChowHeegner} and the work of Gross--Kohnen--Zagier \cite{GKZ} to deduce the non-triviality of the Gross--Kudla--Schoen cycle where the base point is a cusp. 

We start by recalling the construction of Heegner divisors on $X_0(N)$ following \cite[Section IV]{GKZ}.

Let $K=\mathbb{Q}(\sqrt{D})$ be an imaginary quadratic field, where $D<0$ is the discriminant of an order $\mathcal{O}_D$. Assume $(D,N)=1$ and that there exists a class $r \bmod 2N$ satisfying $r^2 \equiv D \bmod 4N$. 
Then we can express the order $\mathcal{O}_D=\mathbb{Z}+\mathbb{Z} \frac{r+\sqrt{D}}{2}$ and $r$ determines a primitive ideal $\mathfrak{n} = \mathbb{Z}N+\mathbb{Z}\frac{r+\sqrt{D}}{2} \subset \mathcal{O}_D$ of index $N$. 

A point $x \in Y_0(N)$ represents a  degree $N$ cyclic isogeny $\pi:E\to E'$ between elliptic curves.  Consider those $x$ for which $\mathfrak{n}$ annihilates $\ker(\pi)$ and there are embeddings $\mathcal{O}_D \hookrightarrow \End(E)$ and $\mathcal{O}_D \hookrightarrow \End(E')$ which make 
\begin{center}
        \begin{tikzcd}
            E \arrow[r,"\pi"] \arrow[d, "\mathcal{O}_D"] & E'  \arrow[d, "\mathcal{O}_D"]\\
            E \arrow[r,"\pi"]& E'.
        \end{tikzcd}
    \end{center}
commute.  The \textit{Heegner divisor} $P_{D,r}$ associated to the data $D$ and $r$ is defined as the sum of such points $x$ with multiplicities $1/e$, where $e$ is the order of $\Aut(x)/\{\pm 1\}$. The divisor $P_{D,r}$ is defined over $\mathbb{Q}(\sqrt{D})$.

Let $\omega_N$ be the $N$-th Atkin--Lehner involution on $X_0(N)$, and let $X_0^*(N)$ denote the quotient of $X_0(N)$ by $\omega_N$. Under this involution, we have $\omega_N(P_{D,r})=P_{D,-r}$ and we use $P_{D,r}^*$ to denote the image of $P_{D,r}$ under the quotient map $X_0(N) \to X_0^*(N)$.
The degree of $P_{D,r}$ and $P_{D,-r}$ is equal to $H_D$, the Hurwitz class number. 

Let $\infty \in X_0(N)$ be the rational cusp given by the image of $\infty \in \overline{\BH}$, and let $\infty^*$ denote the image of $\infty\in X_0(N)$ under the quotient map $X_0(N) \to X_0^*(N)$. Then $P_{D,r}^*-H_D\infty^*$ is a degree $0$ divisor on $X_0^*(N)$ defined over the rational numbers $\mathbb{Q}$, which gives a point $y_{D,r}^*$ on $J_0^*(N)$, the Jacobian of $X_0^*(N)$. The point $y_{D,r}^*$ is the image of $y_{D,r} = P_{D,r}+P_{D,-r}-2H_D\infty$ under the map $J_0(N) \to J_0^*(N)$.

Let $f \in \mathcal{S}^{\rm new}_2(\Gamma_0(N))^-$ be a weight $2$ normalized new form with root number (sign of functional equation)  $-1$\footnote{This is equivalent to $f\,d\tau\in \Omega^1(X_0(N))$ being the pullback of a form in $\Omega^1(X_0^*(N))$.}. By the main theorem of Gross--Zagier \cite[Theorem (6.3)]{GZ} (also see Gross--Kohnen--Zagier \cite[Page 557]{GKZ}), the canonical height $h_K$ of a projection of $y_{D,r}= P_{D,r}-H_D\infty \in J(N)$ over $K=\mathbb{Q}(\sqrt{D})$ is given by
\[ L'(f,1) L(f_D, 1)= \frac{8 \pi^2 (f,f)}{u^2D^{1/2}}h_K(y_{D,r,f}) \]
where $f_D $ is the quadratic twist of $f$ by the Dirichlet character associated to $K/\mathbb Q$, $(f,f)$ is the Peterson inner product, $2u$ is the number of roots of unity in $\mathcal{O}_D$, and $y_{D,r,f}$ is the projection of $y_{D,r}$ to the $f$-isotypical component of $J_0(N)$ as described in \cite[Page 230]{GZ}. Thus, if $L'(f,1) \ne 0$ and $L(f_D, 1) \ne 0$, then $y_{D,r} \in J_0(N)(K)$ is a non-torsion point. In such a case, $y_{D, r, f}$ descends to a non-torsion point in $J_0^*(N)$. It is well-known that there exists infinitely many  $D$ (can be assumed to be fundamental) with $L(f_D, 1) \ne 0$, see for example \cite{Wald}, \cite{BFH}, and \cite{MM}.

\begin{lem}\label{lem:nontorsionHeegner}
    Suppose there exists $f \in \mathcal{S}^{\rm new}_2(\Gamma_0(N))^-$ satisfying $L'(f,1) \ne 0$.  Then there is a special divisor $Z_{X_0(N)}(m_0,\mu_0)\subset X_0(N)$ (as described in Section \ref{sec:explicitcover}) of degree $=:d$,  the degree $0$ divisor $Z_{X_0(N)}(m_0,\mu_0) - d (\infty) \in J_0(N)$ is non-torsion.
\end{lem}

\begin{proof}
By \cite[Theorem C]{GKZ}, there exists a Heegner divisor $y_{D,r}^* \in J^*_0(N)$ such that the $f$-isotypical component has nonzero height. Since $2y^*_{D,r}$ is the image of $P_{D,r}+P_{D,-r}-2H_D\infty$ under the map $J_0(N) \to J_0^*(N)$, it follows that $P_{D,r}+P_{D,-r}-2H_D\infty \in J_0(N)$ is non-torsion.

By Remark \ref{rmk:Heegner}, by taking $D=-4Nm_0$ to be a negative discriminant of an order $\mathcal{O}_D $ in an imaginary quadratic field $K=\mathbb{Q}(\sqrt{D})$ and $r \in \mathbb{Z}/2N\mathbb{Z}$ to satisfy $r^2 \equiv D \bmod 4N$, we obtain the special divisor $Z_{Y_0(N)}(m_0,\mu_0)$ where $\mu_0=\mu_r \in L_W'/L_W$. As was discussed in \cite[Section 7.1]{BY}, following the definition of $Z_{Y_0(N)}(m_0,\mu_0)$, we have 
$$
Z_{X_0(N)}(m_0, \mu_0) = Z_{Y_0(N)}(m_0,\mu_0)=P_{D,r}+P_{D,-r}. $$
We conclude the statement.
\end{proof}

\subsection{Proof of nontriviality of the Gross--Kudla--Schoen cycle with a cusp as base point.}\label{subsec:GKSinfty}
In this section, we prove the cycle $\Delta_{GKS}(X_N,\infty)$ is nontrivial in $\CH^1(X_N^3)$. In section \ref{S5.4}, we give a proof of the fact that the cusp $\infty \in \CH^1(X_N)$ satisfies $(2g-2)\infty=K_{X_N} \in \CH^1(X_N)$. Thus, combined with Example \ref{eg:diagonal}, we obtain Theorem \ref{thm:main2} regarding the Gross--Kudla--Schoen cycle.

\begin{thm}\label{th-main}
If there exists a normalized weight $2$ new form $f \in \mathcal{S}^{\rm new}_2(\Gamma_0(N))^-$ satisfying $L'(f,1) \ne 0$, then the cycle $\Delta_{GKS}(X_N,\infty)$ is nontrivial in in $\CH^2(X_N^3)$ where $X_N$ is the modular curve associated to the congruence subgroup $\Gamma_1(2N)\cap\Gamma(2) \subset \SL_2(\BZ)$.
\end{thm}

\begin{proof}
Recall the construction of $X_N$ and the following commutative diagram from Section \ref{sec:explicitcover}, where the maps $\iota$ and $\iota_1$ are the diagonal embeddings given by the embedding of Hermitian domains $\BD_W \hookrightarrow \BD_V$ discussed in Section \ref{sec:explicitcover}.

  \begin{center}
    \begin{tikzcd}
     Y_N \arrow[r, hook] \arrow[d, "\pi_W"] &  X_N 
\arrow[d,"\pi_W"] \arrow[r, hook, "\iota_1"] &   X_N \times X_N \arrow[d,"\pi_V"] \\
Y_0(N) \arrow[r,hook] & X_0(N) \arrow[r, hook,"\iota"] & X_0(N) \times X_0(N)
    \end{tikzcd}
\end{center}

Following Lemma \ref{lem:nontorsionHeegner}, with our assumption on the existence of $f$, there exists a special divisor $Z_{X_0(N)}(m_0,\mu_0) $ such that the divisor $Z_{X_0(N)}(m_0,\mu_0)  - d\infty \in J_0(N)$ is non-torsion. 

Moreover, by construction of $X_N$ and the special divisors $Z_{X_N}(m_0,\mu_0)$, we have $Z_{X_N}(m_0,\mu_0) = \pi_W^*(Z_{X_0(N)}(m_0,\mu_0))$ and thus $\deg (\pi_W) \cdot Z_{X_0(N)}(m_0,\mu_0) = \pi_{W,*}(Z_{X_N}(m_0,\mu_0))$, which implies the divisor $Z_{X_N}(m_0,\mu_0)  - d_1 \infty \in \Jac(X_N)$ is non-torsion (where $d_1=\deg Z_{X_N}(m_0,\mu_0) $ and $\infty \in X_N$ is the cusp from $\infty$). 

By Proposition \ref{prop:mainprop}, the divisor $Z_{X_N}(m_0,\mu_0)= \iota_1^*(Z)$ for a linear combination of special divisors $Z \in \CH^1(X_N \times X_N)$.

We consider the Gross--Kudla--Schoen cycle $\Delta_{GKS}(X_N,\infty)$ associated to the pointed curve $(X_N,\infty)$.
By Lemma \ref{lem222}, we get the Chow--Heegner divisor
\[\Pi_Z(\Delta_{GKS}(X_N,\infty))= \jmath^*_{12} (Z)- \jmath_1^{\infty,*}(Z)-\jmath_2^{\infty,*}(Z)- \deg(\jmath^*_{12} (Z)- \jmath_1^{\infty,*}(Z)-\jmath_2^{\infty,*}(Z))\infty.\]
The first term satisfies $\jmath^*_{12} (Z)=\iota_1^*(Z)=Z_{X_N}(m_0,\mu_0)$. By Lemma \ref{lem:boundary}, the next two terms $\jmath_1^{\infty,*}(Z)$ and $\jmath_2^{\infty,*}(Z)$ are supported on cusps. By the theorem of Drinfeld \cite{Drinfeld} and Manin \cite{Manin}, the divisor $\Pi_Z(\Delta_{GKS}(X_N,\infty)) = Z_{X_N}(m_0,\mu_0)  - d_1 \infty \in \Jac(X_N) \otimes \mathbb{Q}$ and thus is non-torsion. By Lemma \ref{lem:Z}, we deduce our statement.
\end{proof}

As mentioned in the introduction, the existence of  Hecke eigenform $f$ of weight $2$ and level $N$ with $L'(f, 1) \ne 0$ is weak for our purpose.  Indeed, we have the following theorem.

\begin{thm}  \label{thm-concrete} (a) \, If a positive integer $N$ satisfies one of the following conditions, then  the $\Delta_{GKS}(X_N,\infty)$ is nontrivial in $\CH^2(X_N^3)$:
\begin{enumerate}
    \item $N$ is divisible by a prime $p=37, 43, 53, 61, 67$ or $ p >71$,

    \item   $N$ is divisible by $p^2$ for some prime $ p \ge 11$.
\end{enumerate}

(b) When $N$ is big enough,  then the $\Delta_{GKS}(X_N,\infty)$ cycle is nontrivial in $\CH^2(X_N^3)$. For example 
$$
 N > 2^6 \cdot 3^4 \cdot 5^2 \cdot 7^2 \cdot \prod_{ \begin{matrix} 11 \le p \le 71 \\ p \ne 37, 43, 53, 61, 67 \end{matrix}} p ,
$$
 is enough.   
\end{thm}
\begin{proof}  First notice that if the theorem is true for $M$, then it is true for all multiples of $M$.  Indeed, for $M|N$,   let 
$\pi\colon X_N\twoheadrightarrow X_M$ be the natural quotient map, and write $d$ for its degree.  Then $\pi_*\Delta_{GKS}(X_N,\infty)=d\cdot\Delta_{GKS}(X_M,\infty)$.

(a)(1): Assume first $p |N$, we can assume $N=p$  by the above comment.  By \cite[Proposition 4.5]{Boya}, for primes $p=67,73,97, 103,107, 109, 113, 127$ or $p>131$, we get the genus of $X_0^*(p)$ is at least $2$. By \cite[Theorem 2]{DF},  there exists a normalized weight $2$ new form $f \in \mathcal{S}^{\rm new}_2(\Gamma_0(p))^-$ satisfying $L'(f,1) \ne 0$. The claim follows then  from  Theorem \ref{th-main}.

 For the cases listed in \cite[Proposition 4.5]{Boya} where the genus of $X_0^*(p)$ equals to $1$, that is for $p=37, 43, 53, 61, 79, 83, 89, 101, 131$, we checked from the database LMFDB \cite{lmfdb} that  there exists a unique normalized weight $2$ new form $f \in \mathcal{S}^{\rm new}_2(\Gamma_0(p))^-$ and  $L'(f,1) \ne 0$. The claim is again true.

 (a)(2) \,  By (a)(1) just proved, we just need to consider the primes  
$$p=2, 3, 5, 7, 11, 13, 17, 19, 23, 29, 31, 41, 47, 59, 71,$$
By  the table in \cite[Proposition 4.5]{Boya}, $X_0^*(p^2)$  has genus as least $2$ for  $p \ge 11$. The results \cite[Theorem 2]{DF} and Theorem \ref{th-main} prove these cases too. 

(b)\,   The database LMFDB \cite{lmfdb} shows that there are normalized new forms $f$  of weight $2$ and level $N$ with $L'(f, 1) \ne 0$ for $N=2^7, 3^5, 5^3, 7^3$.  Now Claim (b) is clear. 
\end{proof}

\begin{rmk} Without appealing to database LMFDB \cite{lmfdb}, Theorem \ref{thm-concrete} (b) would follow from the following analytic result: when $N$ is big enough (perhaps effective but not fun to calculate), there is always some new form $f$ of weight $2$ and level $M|N$ with $L'(f, 1) \ne 0$.  When $N$ is restricted to prime, a much stronger theorem was proved in \cite{KM}. The method there should be easily extended to prove the above claim for general $N$ although we did not find precise reference for this extension. On the other extreme, when $-N$ is a fundamental discriminant of an imaginary quadratic field $\mathbb Q(\sqrt{-N})$. Rohrlich  constructed one or two canonical family of CM Hecke eigenforms of  weight  $2$ and level either $N^2$ or $4 N^2$ in  \cite{Roh} (see also \cite{Yang} for $4\| N$).  Montgomery and Rohrlich (\cite{MR}) proved the central $L$-value always non-vanishing when root number is $1$,  and S.D. Miller and Yang (\cite{MY}) proved that the central derivative of the $L$-function is always non-vanishing when the root number is $-1$.
    \end{rmk}

\subsection{Projections and Jacobians}\label{S5.3}

Before commencing the proof that Abel-Jacobi-nontriviality of $\gks(X_N,o)$ is independent of base point $o\in X_N(\CC)$, we introduce some projection morphisms for the Hodge structures and intermediate Jacobians involved.  These will all be defined at the level of algebraic cycles.  Some tedious notational preliminaries are needed, but the resulting self-contained analysis avoids other complications in the literature and should be applicable to other kinds of Abel--Jacobi maps ($\ell$-adic, $p$-adic, etc.).

Recall throughout the paper, all Chow groups and Jacobians are taken $\otimes \QQ$ and thus ``$\neq 0$'' and ``nontrivial'' mean in effect ``is nontorsion'' in the integral Chow groups.
Given a rational Hodge structure $H$ of odd weight $2p-1$, we  define its Jacobian to be the complex torus
$$J(H):=\Extm^1(\QQ(-p),H)\cong \frac{H_{\CC}}{F^p H_{\CC}+H}.$$
Note that when $H=H^1(X)$, the singular cohomology of a curve $X$, we see the classical Jacobian\footnote{Viewing the Jacobian variety $\mathrm{Jac}(X)$ as a group, $J(X)$ means $\mathrm{Jac}(X)\otimes \QQ$.} $J(H^1(X))=J(X)$, and that exact sequences of such Hodge structures produce exact sequences of Jacobians.  As above, we let $X$ be a modular curve of genus $g>2$; we write $o,P,Q,\infty\in X(\CC)$ for points, with $\infty$ a cusp.

Writing $\I\subset \{1,2,3\}$ (and $\hat{\I}$ its complement), let $\pi_{\I}\colon X^3\onto \prod_{i\in \I}X_i$ denote the coordinate projections; so for example $\pi_{23}=\pi_{\hat{1}}$.  Order of coordinates is always preserved.  Let $\imath_{\I}^P\colon X\into \xT$, $\mathcal{I}_i^P\colon \xt\into\xT$, and if $\I\subset\{1,2\}$, $\jmath_{\I}^P\colon X\into \xt$ denote inclusions:
\begin{itemize}[leftmargin=0.5cm]
\item $\imath_{\I}^P(x),\jmath_{\I}^P(x)$ have $i^{\text{th}}$ coordinate $x$ for $i\in \I$ and $P$ for $i\notin \I$; and
\item $\mathcal{I}_i^P$ inserts $P$ in the $i^{\text{th}}$ coordinate.
\end{itemize}
Define cycles $X_{\I}^P:=(\imath_{\I}^P)_*X$ (or $(\jmath_{\I}^P)_* X)$) in $\CH^2(\xT)$ (resp.~$\CH^1(\xt)$), and $$P_i:=(\mathcal{I}_i^P)_*(\xt)=\pi_i^*P\in \CH^1(\xT).$$  So for instance, with this notation we have $P_1P_2=X_3^P$, and 
\begin{equation}\label{en1}
\gks(X,P):=X_{123}-\sum_{i<j}X_{ij}^P-\sum_k X_i^P.
\end{equation}
We will later specialize $P$ to a general point $o \in X$ or the cusp $\infty \in X$.

Next we introduce our main symmetrization and projection maps:
\begin{equation}\label{en2}
\sym:=\tfrac{1}{6}\sum_{g\in\mathfrak{S}_3}g^*\;\;\text{ and }\;\;\prp^P:=\text{id}-\sum_i \pi_{\hat{i}}^*(\mathcal{I}_i^P)^*,
\end{equation}
which may be viewed as endomorphisms of $\CH^*(\xT)$, $H^*(\xT)$, or related Jacobians.  

More precisely, writing $H^1:=H^1(X)$, the endomorphism $\prp^P$ induces a projection\footnote{The meaning of ``projection'' includes being the identity on the target of $\onto$.} $H^3(\xT)\onto (H^1)^{\otimes 3}\subset H^3(\xT)$, while $\sym$ induces $(H^1)^{\otimes 3}\onto \bigwedge^3 H^1\subset (H^1)^{\otimes 3}$.  Their composite $\Sym^P:=\sym\circ \prp^P$ thus projects $H^3(\xT)\onto \bigwedge^3 H^1\subset H^3(\xT)$. By definition, we have $\sym(\gks(X,P))=\gks(X,P)$.

Finally, we need some operations relating $H^1$ and $H^3(\xT)$.  Set $$\theta:=X_{12}-X_1^{\infty}-X_2^{\infty}\in \CH^1(\xt)$$ and $$\tilde{\theta}:=3\cdot\sym (\pi_{12}^*\theta)\in \CH^1(\xT).$$  Then by denoting intersection product (cup-product on the level of cohomology) by ``$\cdot$'', we have
\begin{equation}\label{en3}
\theta_*:= \prp^{\infty}\circ (\tilde{\theta}\cdot)\circ (\sym\circ \pi_1^*)
\end{equation}
which maps $\CH^1(X)\to \CH^2(\xT)$ and embeds $H^1\into \bigwedge^3 H^1$.  

Denote the image of this embedding by $\theta H^1$.  For any divisor $Z\in \CH^1(\xt)$, we get a map
\begin{equation}\label{en4}
Z^*:=\pi_{3*}\circ ((\pi_{12}^*)\cdot)
\end{equation}
which maps $\CH^2(\xT)\to \CH^1(X)$ and $H^3(\xT)\to H^1$.  
We are especially interested in the latter as a morphism $Z^*\colon \bigwedge^3 H^1\to H^1$. 

Recall from Definition \ref{def:ChowHeegner}, for the cycle $\gks(X,o)$, we use notation $\Pi_Z(\gks(X,o)) = Z^*(\gks(X,o))$ and call this divisor the Chow--Heegner divisor.

Define the ``primitive'' sub-Hodge structure 
\begin{equation}\label{en5}
\textstyle\wpr:=\ker(\theta^*)\subset \bigwedge^3 H^1.
\end{equation}

Note that $\AJ$ is functorial with respect to correspondences, and Equations \eqref{en1}-\eqref{en4} are all induced by correspondences.

\begin{lem}\label{l0}
$\AJ(\gks(X,P))\in J(\bigwedge^3H^1)\subset J(H^3(\xT))$.
\end{lem}
\begin{proof}
On $\xt$, the zero-cycle $\mathcal{B}_{P,Q}:=(P,P)-(P,Q)-(Q,P)+(Q,Q)$ is Abel-Jacobi-equivalent to zero (written $\ajeq 0$), since its projections to the two copies of $X$ are zero.  

A short computation gives
$$\Sym^Q(\gks(X,P))=\gks(X,P)-3\cdot\sym\,\pi_{12}^*\mathcal{B}_{P,Q}\ajeq \gks(X,P),$$
whence $\AJ(\gks(X,P))=\Sym^Q\AJ(\gks(X,P))$.  

As $\Sym^Q$ maps $J(H^3(\xT))\onto J(\bigwedge^3H^1)$, the result follows.
\end{proof}

For a divisor $Z \in \CH^1(X^2)$ as above, set $\theta^P*Z:=\jmath_{12}^*Z-(\jmath_1^{P})^*Z-(\jmath_2^P)^*Z$.  

Write $|W|:=\deg(W)$ for any $0$-cycle $W$.

\begin{lem}\label{l4}
$Z^*\gks(X,P)=\theta^P*Z-|\theta\cdot Z|P\in \CH^1_{\text{hom}}(X)$. 
\end{lem}
\begin{proof}
This is just a restatement of Lemma \ref{lem222} with this new notation.
\end{proof}

\subsection{Independence of basepoint}\label{S5.4}
In Section \ref{subsec:GKSinfty}, we showed that the cycle $\gks(X_N,\infty)$ is nontrivial in $\CH^1(X^3_N)$. In this section, we will show that this nontriviality statement holds with the base point $\infty$ replaced by any base point, which concludes Theorem \ref{thm:main2}. Moreover, we will show the same nontriviality statement holds for the Ceresa cycle $X^+_o-X^-_o \in \CH^{g-1}(\Jac(X))$ associated to $(X_N,o)$ where $o$ is any basepoint, which concludes Theorem \ref{thm:main1}.

In this subsection we assume that $X=X(\Gamma)$ is modular with $\Gamma$ \emph{torsion-free}.  (In particular, we can take $X=X_N$.) 

Consider the diagonal $X_{12}=\jmath_{12}(X)\subset \xt$, its normal bundle $\mathcal{N}$, and its self-intersection $\jmath_{12}^* X_{12}=c_1(\mathcal{N})\in \CH^1(X)$ in $\xt$ as a cycle on $X$.  Since $\omega_{\xt}|_{X_{12}}\cong \omega_X^{\otimes 2}$, adjunction gives $\mathcal{N}\cong \omega_X\otimes \omega_{\xt}^{-1}|_{X_{12}}\cong \omega_X^{-1}$.  Writing $K_X$ for a canonical divisor, the self-intersection is thus
\begin{equation}\label{en6}
\jmath_{12}^*X_{12}=-K_X.	
\end{equation}
This immediately yields $|\theta^2|=-2g$, and thus (by Lemma \ref{l4}) that
\begin{equation}\label{en7}
\theta^*\gks(X,P)=\theta^P*\theta-|\theta^2|P=-K_X+(2g-2)P.
\end{equation}
Recall $\theta^*\gks(X,P)=\Pi_{X_{12}}(\gks(X,P))$ and the computation above shows the statement in Example \ref{eg:diagonal}.

In the following Lemma, we show that the Abel--Jacobi image of $\gks(X,\infty)$ lies in the primitive part of the intermediate Jacobian. We first show by computation that when the base point is an element in $\CH^1(X)$ satisfying $(2g-2)o=K_X \in \CH^1(X)$, the Abel--Jacobi image of $\gks(X,o)$ lies in the primitive part $J(\wpr)$. Then we show that for modular curves associated to a torsion-free congruence subgroup $\Gamma$, the cusp $\infty$ is such an element in $\CH^1(X)$.

\begin{lem}\label{l2}
$\AJ(\gks(X,\infty))\in J(\wpr)$.
\end{lem}
\begin{proof}
Setting $P=\infty$ in \eqref{en7} gives 
$$\theta^*\AJ(\gks(X,\infty))=\AJ(\theta^*\gks(X,\infty))=\AJ(-K_X+(2g-2)\infty).$$ 
By \eqref{en5}, it suffices to show that this is zero.  Indeed, we will show that $K_X-D$ is zero in $\CH^1(X)$ for any divisor $D$ of degree $2g-2$ supported on the cusps.

Let $Y\subset X$ be the complement of the cusps, and write $Y\cong \Gamma\backslash\BH$ with $\Gamma\leq \mathrm{SL}_2(\ZZ)$.  As a function on the upper half-plane, the discriminant $\Delta$ is a constant multiple of $g_2^3-27g_3^2$, hence is nowhere vanishing. It is also an element of $S_{12}(\Gamma)$, so that $\Delta(\tau)d\tau^{\otimes 6}\in H^0(Y,\omega_X^{\otimes 6}|_Y)$ is a nowhere vanishing section trivializing the restriction of $\omega_X^{\otimes 6}\cong \mathcal{O}_X(6K_X)$ to $Y$. Note that if $Y$ has elliptic points, then $d\tau^{\otimes 6}$ has nontrivial divisor as a section of $\omega_X^{\otimes 6}|_Y$.  So the torsion-free hypothesis on $\Gamma$ is really needed here.

By exactness of the localization sequence
$$\CH^0(X\setminus Y)\to \CH^1(X)\to \CH^1(Y)\to 0$$
it follows that $6K_X\underset{\text{rat}}{\equiv}E$ for some $E\in \text{Div}(X)$ supported on the cusps.  So $6(K-D)=6K-6D\underset{\text{rat}}{\equiv}E-6D$ is of degree $0$ and supported on the cusps, hence torsion (i.e.~zero in the rational Chow group) by the work of Manin\cite{Manin} and Drinfeld \cite{Drinfeld}.
\end{proof}

Following from the statement that the cusp $\infty\in\CH^1(X)$ satisfies $(2g-2)\infty=K_X\in\CH^1(X)$ and Example \ref{eg:diagonal}, we could already conclude that if the cycle $\gks(X,\infty)$ is nontrivial in $\CH^2(X^3)$, then the same nontriviality statement holds for any basepoint which concludes Theorem \ref{thm:main2}. In the following statements Lemma \ref{l1}, Lemma \ref{l3}, and Proposition \ref{pr24}, we give a proof of the independence of base point from a Hodge-theoretic perspective. Namely, the Abel--Jacobi images of the cycles $\gks(X,P)-\gks(X,Q)$ where $P,Q \in \CH^1(X)$ are any two degree $1$ divisors and $\gks(X,\frac{1}{2g-2}K_X)$ lie in different direct components of the intermediate Jacobian. Thus, the nontriviality of $\gks(X,\frac{1}{2g-2}K_X)$ would not be ``killed'' by changing a base point. This Hodge-theoretic argument will allow us to deduce the nontriviality of the Ceresa cycle from the nontriviality of the Gross--Kudla--Schoen cycle and conclude Theorem \ref{thm:main1}.

\begin{lem}\label{l1}
We have the splitting of rational Hodge structures $\bigwedge^3H^1=\wpr\oplus \;\theta H^1$, and of Jacobians $J(\bigwedge^3H^1)=J(\wpr)\oplus J(\theta H^1)$ \textup{(}where $J(\theta H^1)\cong J(X)$\textup{)}.
\end{lem}
\begin{proof}
We must show that $\theta^*\circ \theta_*$ induces an isomorphism.  Since the two assertions in the Lemma are equivalent, and ($\AJ$ of) linear combinations of zero-cycles $P-\infty$ span $J(X)$, it is enough to check that $\theta^*\theta_*(P-\infty)\ajeq\tfrac{2g-2}{3}(P-\infty)$ for all $P$.  Indeed, we compute
\begin{align*}
3\,\theta_*(P-\infty)&=\textstyle\prp^{\infty}\{(\sum_k(P_k-\infty_k))\cdot (\sum_{i<j}(\pi_{ij}^*X_{ij}-\infty_j-\infty_i))\}\\
&=\textstyle 3\,\Sym^{\infty}\{(X_{ij}^P-X_{ij}^{\infty})+2\mathcal{B}_{P,\infty}\}\ajeq 3\,\Sym^{\infty}(X_{ij}^P-X_{ij}^{\infty})\\
&=\textstyle \prp^{\infty}\sum_{i<j}(X_{ij}^P-X_{ij}^{\infty})\\
&=\textstyle\sum_{i<j}(X_{ij}^P-X_{ij}^{\infty}+2\infty_i\infty_j-\infty_i o_j-o_i\infty_j)\\
&\textstyle\ajeq \sum_{i<j}(X_{ij}^P-X_{ij}^{\infty})+\sum_k (X_i^{\infty}-X_i^P)=\gks(X,P)-\gks(X,\infty), 
\end{align*}
and then applying $\theta^*$ gives $(-K_X+(2g-2)P)-(-K_X+(2g-2)\infty)=(2g-2)(P-\infty)$ by \eqref{en7}.
\end{proof}

\begin{lem}\label{l3}
$\AJ(\gks(X,o)-\gks(X,\infty))\in J(\theta H^1)$.
\end{lem}
\begin{proof}
By the calculation in the proof of Lemma \ref{l1}, $$\tfrac{1}{3}\AJ(\gks(X,o)-\gks(X,\infty))=\AJ(\theta_*(o-\infty))=\theta_*(\AJ(o-\infty)). \qedhere $$
\end{proof}

\begin{prop}\label{pr24}
If $\AJ(\gks(X,\infty))\neq 0$, then $\AJ(\gks(X,o))\neq 0$ for all $o$.
\end{prop}
\begin{proof}
Obviously $\gks(X,o)=\gks(X,\infty)+(\gks(X,o)-\gks(X,\infty))$, and by Lemmas \ref{l2}-\ref{l3}, the $\AJ$-images of these two terms belong to complementary subgroups of $J(\bigwedge^3 H^1)$.
\end{proof}

\begin{proof}[Proof of \ref{thm:main1} and \ref{thm:main2}]
By Lemma \ref{l4}, $Z^*\gks(X,\infty)$ recovers $\Pi_Z(\gks(X,\infty))$.  The proof of Theorem \ref{th-main} produces $Z\in \CH^1(\xt)$ such that $Z^*\AJ(\gks(X,\infty))=\AJ(\Pi_Z(\gks(X,\infty)))\neq 0$, whence $\AJ(\gks(X,\infty)))\neq 0$.  Theorem \ref{thm:main2} now follows from Proposition \ref{pr24}.

Fix a base point $o$ and define $f\colon X^{\times 3}\to \mathrm{Jac}(X)$ by $(q_1,q_2,q_3)\mapsto\mathrm{AJ}(\sum q_i - 3o)$. By \cite[Prop.~2.9]{CG}, $f_*\gks(X,o))$ and $3(X_o^+-X_o^-)$ have the same Abel-Jacobi images after projecting to $J_{\text{pr}}:=J(H^{2g-3}_{\text{prim}}(\mathrm{Jac}(X))(g-1))$; while $f_*$ induces an isomorphism from $J(\wpr)$ to $J_{\text{pr}}$ under which the primitive AJ-images of $\gks(X,o)$ and $f_*\gks(X,o)$ agree.  Since the primitive AJ-image of $\gks(X,o)$ is nonzero and independent of $o$ by the lemmas above, so is that of $X_o^+-X_o^-$.  This proves Theorem \ref{thm:main2}.
\end{proof}

With the same argument, Theorem \ref{Ceresa-concrete} stated in the introduction follows from Theorem \ref{thm-concrete}.

\subsection{Remarks on Abel-Jacobi maps}\label{S5.5}
In the previous subsection, only the first of the three Lemmas is special to certain modular curves.  For a more general curve (of genus at least $3$), Proposition \ref{pr24} generalizes to the statement that
\begin{equation}\label{en8}
\textit{The projection of $\AJ(\gks(X,o))$ to $\textstyle J(\wpr)$ is independent of $o$,}
\end{equation}
where we mean the projection $\mathrm{pr}\colon J(\bigwedge^3 H^1)\onto J(\wpr)$ killing $J(\theta H^1)$. This projection is induced by the operator $$\mathrm{pr}:=\text{id}-\tfrac{3}{2g-2}\theta_*\theta^*\in \mathrm{End}(\CH^3(\xT))$$
since $\theta^*\mathrm{pr}(\W)=\theta^*\W-\tfrac{3}{2g-2}\theta^*\theta_*\theta^*\W=\theta^*\W-\theta^*\W=0$ for any $\W\in \CH^2(\xT)$.

Recall from \eqref{en4} that for any $Z\in \CH^1(\xt)$, $Z^*$ maps $\CH^2(\xT)\to \CH^1(X)$.  The induced map of Jacobians $Z^*\colon J(\bigwedge^3 H^1)\onto J(X)$ only depends on the (Hodge) class\footnote{Here we take $[Z]$ to mean the projection of the fundamental class from $H^2(\xt)\onto \bigwedge^2 H^1$.} $[Z]\in \mathrm{Hg}^1(\bigwedge^2 H^1)\in \mathrm{Hom}_{\text{MHS}}(\QQ(-1),\bigwedge^2 H^1)$, inducing pairings $\mathrm{Hg}^1(\bigwedge^2 H^1)\times J(\bigwedge^3 H^1)\to J(X)$ and 
\begin{equation}\label{en9}
\langle\;\cdot\;,\;\cdot\;\rangle\colon \frac{\mathrm{Hg}^1(\bigwedge^2 H^1)}{\langle[\theta]\rangle}\times J(\textstyle\wpr)\to J(X)
\end{equation}
An immediate consequence of \eqref{en8} is then that
\begin{equation}\label{en10}
\langle\;\cdot\;,\AJ(\mathrm{pr}(\gks(X,o)))\rangle\in \textstyle\mathrm{Hom}\left(\mathrm{Hg}^1(\bigwedge^2 H^1)/\langle[\theta]\rangle,J(X)\right)
\end{equation}
is independent of $o$.\footnote{Indeed, an easy calculation gives $Z^*\mathrm{pr}(\gks(X,o))=\tfrac{-1}{2g-2}Z^*\gks^{K_X}$, where the notation $\gks^{\cdots}$ has been linearly extended to divisors on $X$.} This is precisely the invariant which we have shown to be nonzero for most modular curves.  In contrast, for a \emph{general} curve, \eqref{en9}-\eqref{en10} are zero since $\mathrm{Hg}^1(\bigwedge^2 H^1)$ is generated by $[\theta]$; moreover, $J(\wpr)$ is irreducible, so cannot in any case map nontrivially to $J(X)$.

For one more perspective on the situation, let $F^1_h (\bigwedge^3 H^1) \subset \bigwedge^3 H^1$ denote the maximal sub-Hodge structure of level $1$, i.e.~with complexification contained in $F^1(\bigwedge^3 H^1_{\CC})$.  Write $\bigwedge^3 H^1=F^1_h(\bigwedge^3 H^1)\oplus \bigwedge^3_{G}=\theta H^1\oplus V(-1)\oplus \bigwedge^3_G$, where $\bigwedge^3_G$ has level $3$ with no level-$1$ sub-HS, $V$ is a weight-1 Hodge structure, and $\wpr=V(-1)\oplus \bigwedge^3_G$.  (For $X$ general, we have $V=\{0\}$.)  Consider the restriction
\begin{equation}\label{en11}
\AJ\colon \Sym^o \CH^3(\xT)\to J(\textstyle\bigwedge^3 H^1)=J(X)\times J(V)\times J(\textstyle\bigwedge^3_G)
\end{equation}
of the Abel-Jacobi map to ``symmetric'' cycles, and write accordingly $$\AJ(\W)=(\AJ(\W)_J,\AJ(\W)_V,\AJ(\W)_G).$$  Then:
\begin{itemize}[leftmargin=0.5cm]
\item $\AJ(\W)_G$ detects nontriviality of $\W$ in the Griffiths group of cycles modulo \emph{algebraic} equivalence, but is hard to compute;\footnote{In particular, it is killed by $Z^*$, since the only morphism from $\bigwedge^3_G$ to $H^1(-1)$ (hence from $J(\bigwedge^3_G)$ to $J(X)$) is zero; otherwise $\bigwedge^3_G$ would have a nontrivial level-1 summand.}
\item $\AJ(\W)_V$ only detects nontriviality of $\W$ modulo rational equivalence, but is nonzero whenever $0\neq Z^*\AJ(\mathrm{pr}(\W)))\in J(X)$ for some $Z$;\footnote{Note that the ``pr'' here is essential.} and
\item for $\W=\gks(X,o)$, $\AJ(\W)_G$ and $\AJ(\W)_V$ are independent of basepoint $o$.
\end{itemize}
One can also ask about the converse of the second bullet:  
\begin{equation}\label{en12}
\textit{If $\AJ(\W)_V\neq 0$, can this always be detected by applying some $Z^*$?}
\end{equation}
A sufficient criterion for an affirmative answer to \eqref{en12} is to have 
\begin{equation}\label{en13}
F^1_h({\textstyle \bigwedge^3 H^1})\;\;\;= \sum_{Z\in \mathrm{Hg}^1(\wedge^2 H^1)} \QQ\langle [Z]\rangle \wedge H^1,
\end{equation}
since then the ``test forms'' for the first two factors of \eqref{en11} are of the form $\Sym(\delta_{\pi_{12}^*Z}\wedge \pi_3^*\omega)$, where $\delta_{\cdots}$ denotes the current of integration and $\omega\in \Omega^1(X)$.

Define a level-$1$ Hodge structure $H$, and its Jacobian $J(H)$, to be \emph{stably nondegenerate} if the ring of Hodge classes in $H^*(J(H)^{\times k})$ is generated by divisor classes for every $k\geq 1$.  It follows from the proof of the main theorem of \cite{Ha} that if $J(X)$ is stably nondegenerate, then \eqref{en12}-\eqref{en13} are true. Moreover, according to \cite[Thm.~1.2]{Ha2}, $J(X)$ is stably nondegenerate for the modular curves $X=X_1(N)$ and any curves dominated by them, such as $X_0(N)$.  Thus we expect the basic approach via Lemma \ref{lem:Z} to studying the GKS and Ceresa cycles to have broader scope than just for the curves $X_N$ treated here.




\begin{bibdiv}
	\begin{biblist}

 \bib{Boya}{article}{
   author={Ad\v zaga, Nikola},
   author={Arul, Vishal},
   author={Beneish, Lea},
   author={Chen, Mingjie},
   author={Chidambaram, Shiva},
   author={Keller, Timo},
   author={Wen, Boya},
   title={Quadratic Chabauty for Atkin-Lehner quotients of modular curves of
   prime level and genus 4, 5, 6},
   journal={Acta Arith.},
   volume={208},
   date={2023},
   number={1},
   pages={15--49},
   issn={0065-1036},
   review={\MR{4616870}},
   doi={10.4064/aa220110-7-3},
}

 \bib{BS}{article}{
   author={Beauville, Arnaud},
   author={Schoen, Chad},
   title={A non-hyperelliptic curve with torsion Ceresa cycle modulo
   algebraic equivalence},
   journal={Int. Math. Res. Not. IMRN},
   date={2023},
   number={5},
   pages={3671--3675},
   issn={1073-7928},
   review={\MR{4565651}},
   doi={10.1093/imrn/rnab344},
}

 \bib{Bloch}{article}{
   author={Bloch, Spencer},
   title={Algebraic cycles and values of $L$-functions},
   journal={J. Reine Angew. Math.},
   volume={350},
   date={1984},
   pages={94--108},
   issn={0075-4102},
   review={\MR{0743535}},
   doi={10.1515/crll.1984.350.94},
}


 \bib{BY}{article}{
   author={Bruinier, Jan Hendrik},
   author={Yang, Tonghai},
   title={Faltings heights of CM cycles and derivatives of $L$-functions},
   journal={Invent. Math.},
   volume={177},
   date={2009},
   number={3},
   pages={631--681},
   issn={0020-9910},
   review={\MR{2534103}},
   doi={10.1007/s00222-009-0192-8},
}

\bib{BFH}{article}{
   author={Bump, Daniel},
   author={Friedberg, Solomon},
   author={Hoffstein, Jeffrey},
   title={Nonvanishing theorems for $L$-functions of modular forms and their
   derivatives},
   journal={Invent. Math.},
   volume={102},
   date={1990},
   number={3},
   pages={543--618},
   issn={0020-9910},
   review={\MR{1074487}},
   doi={10.1007/BF01233440},
}

\bib{Ceresa}{article}{
   author={Ceresa, G.},
   title={$C$ is not algebraically equivalent to $C\sp{-}$ in its Jacobian},
   journal={Ann. of Math. (2)},
   volume={117},
   date={1983},
   number={2},
   pages={285--291},
   issn={0003-486X},
   review={\MR{0690847}},
   doi={10.2307/2007078},
}

\bib{CG}{article}{
   author={Colombo, Elisabetta},
   author={van Geemen, Bert},
   title={Note on curves in a Jacobian},
   journal={Compositio Math.},
   volume={88},
   date={1993},
   number={3},
   pages={333--353},
   issn={0010-437X},
   review={\MR{1241954}},
}

\bib{DDLR}{article}{
   author={Darmon, Henri},
   author={Daub, Michael},
   author={Lichtenstein, Sam},
   author={Rotger, Victor},
   title={Algorithms for Chow-Heegner points via iterated integrals},
   journal={Math. Comp.},
   volume={84},
   date={2015},
  number={295},
   pages={2505--2547},
   issn={0025-5718},
   review={\MR{3356037}},
   doi={10.1090/S0025-5718-2015-02927-5},
}

 \bib{DRS}{article}{
   author={Darmon, Henri},
   author={Rotger, Victor},
   author={Sols, Ignacio},
   title={Iterated integrals, diagonal cycles and rational points on
   elliptic curves},
   language={English, with English and French summaries},
   conference={
      title={Publications math\'{e}matiques de Besan\c{c}on. Alg\`ebre et th\'{e}orie des
      nombres, 2012/2},
   },
   book={
      series={Publ. Math. Besan\c{c}on Alg\`ebre Th\'{e}orie Nr.},
      volume={2012/},
      publisher={Presses Univ. Franche-Comt\'{e}, Besan\c{c}on},
   },
   date={2012},
   pages={19--46},
   review={\MR{3074917}},
}

\bib{DF}{article}{
   author={Dogra, Netan},
   author={Le Fourn, Samuel},
   title={Quadratic Chabauty for modular curves and modular forms of rank
   one},
   journal={Math. Ann.},
   volume={380},
   date={2021},
   number={1-2},
   pages={393--448},
   issn={0025-5831},
   review={\MR{4263688}},
   doi={10.1007/s00208-020-02112-3},
}

 \bib{Drinfeld}{article}{
   author={Drinfeld, V. G.},
   title={Two theorems on modular curves},
   language={Russian},
   journal={Funkcional. Anal. i Prilo\v{z}en.},
   volume={7},
   date={1973},
   number={2},
   pages={83--84},
   issn={0374-1990},
   review={\MR{0318157}},
}


\bib{EM}{article}{
   author={Eskandari, Payman},
   author={Murty, V. Kumar},
   title={On Ceresa cycles of Fermat curves},
   journal={J. Ramanujan Math. Soc.},
   volume={36},
   date={2021},
   number={4},
   pages={363--382},
   issn={0970-1249},
   review={\MR{4355369}},
}

\bib{EMPAMS}{article}{
   author={Eskandari, Payman},
   author={Murty, V. Kumar},
   title={On the harmonic volume of Fermat curves},
   journal={Proc. Amer. Math. Soc.},
   volume={149},
   date={2021},
  number={5},
   pages={1919--1928},
   issn={0002-9939},
   review={\MR{4232186}},
   doi={10.1090/proc/15332},
}


\bib{GKZ}{article}{
   author={Gross, B.},
   author={Kohnen, W.},
   author={Zagier, D.},
   title={Heegner points and derivatives of $L$-series. II},
   journal={Math. Ann.},
   volume={278},
   date={1987},
   number={1-4},
   pages={497--562},
   issn={0025-5831},
   doi={10.1007/BF01458081},
}

\bib{GK}{article}{
   author={Gross, Benedict H.},
   author={Kudla, Stephen S.},
   title={Heights and the central critical values of triple product
   $L$-functions},
   journal={Compositio Math.},
   volume={81},
   date={1992},
   number={2},
   pages={143--209},
   issn={0010-437X},
   review={\MR{1145805}},
}

\bib{GS}{article}{
   author={Gross, B. H.},
   author={Schoen, C.},
   title={The modified diagonal cycle on the triple product of a pointed
   curve},
   language={English, with English and French summaries},
   journal={Ann. Inst. Fourier (Grenoble)},
   volume={45},
   date={1995},
   number={3},
   pages={649--679},
   issn={0373-0956},
   review={\MR{1340948}},
   doi={10.5802/aif.1469},
}

\bib{GZ}{article}{
   author={Gross, Benedict H.},
   author={Zagier, Don B.},
   title={Heegner points and derivatives of $L$-series},
   journal={Invent. Math.},
   volume={84},
   date={1986},
   number={2},
   pages={225--320},
   issn={0020-9910},
   doi={10.1007/BF01388809},
}

\bib{Harris}{article}{
   author={Harris, Bruno},
   title={Homological versus algebraic equivalence in a Jacobian},
   journal={Proc. Nat. Acad. Sci. U.S.A.},
   volume={80},
   date={1983},
   number={4},
   pages={1157--1158},
   issn={0027-8424},
   review={\MR{0689846}},
   doi={10.1073/pnas.80.4.1157},
}

\bib{Ha2}{article}{
   author={Hazama, Fumio},
   title={Algebraic cycles on abelian varieties with many real endomorphisms},
   journal={T\^ohoku Math.~J.},
   volume={35},
   date={1983},
   pages={303-308},
   review={\MR{0699932}},
   doi={},
}

\bib{Ha}{article}{
   author={Hazama, Fumio},
   title={The generalized Hodge conjecture for stably nondegenerate abelian varieties},
   journal={Compositio Math.},
   volume={93},
   date={1994},
   number={2},
   pages={129-137},
   issn={},
   review={\MR{1287693}},
   doi={},
}



\bib{KowalskiThesis}{book}{
   author={Kowalski, Emmanuel Guillaume},
   title={The rank of the Jacobian of modular curves: Analytic methods},
 note={Thesis (Ph.D.)--Rutgers The State University of New Jersey - New
   Brunswick},
   publisher={ProQuest LLC, Ann Arbor, MI},
   date={1998},
   pages={209},
   isbn={978-0591-97579-6},
   review={\MR{2697914}},
}

\bib{Kud97}{article}{
   author={Kudla, Stephen S.},
   title={Algebraic cycles on Shimura varieties of orthogonal type},
   journal={Duke Math. J.},
   volume={86},
   date={1997},
   number={1},
   pages={39--78},
   issn={0012-7094},
   doi={10.1215/S0012-7094-97-08602-6},
}

\bib{KM}{article}{
    AUTHOR = {Kowalski, E.},
    author={Michel, P.},
     TITLE = {A lower bound for the rank of {$J_0(q)$}},
   JOURNAL = {Acta Arith.},
  FJOURNAL = {Acta Arithmetica},
    VOLUME = {94},
      YEAR = {2000},
    NUMBER = {4},
     PAGES = {303--343},
      ISSN = {0065-1036,1730-6264},
   MRCLASS = {11F67 (11F30 11F66 11G40)},
  MRNUMBER = {1779946},
MRREVIEWER = {M.\ Ram\ Murty},
       DOI = {10.4064/aa-94-4-303-343},
       URL = {https://doi.org/10.4064/aa-94-4-303-343},
}

\bib{LS}{article}{
   author={Lilienfeldt, David T.-B. G.},
   author={Shnidman, Ari},
   title={Experiments with Ceresa classes of cyclic Fermat quotients},
   journal={Proc. Amer. Math. Soc.},
   volume={151},
   date={2023},
   number={3},
   pages={931--947},
   issn={0002-9939},
   review={\MR{4531629}},
   doi={10.1090/proc/16178},
}

\bib{lmfdb}{article}{
  shorthand    = {LMFDB},
  author       = {The {LMFDB Collaboration}},
  title        = {The {L}-functions and modular forms database},
  howpublished = {\url{https://www.lmfdb.org}},
  year         = {2024},
  note         = {[Online; accessed 2 July 2024]},
}

\bib{LR}{article}{
    author={Lupoian, Elvira},
    author={Rawson, James},
    title={Ceresa Cycles of $X_0(N)$},
    journal={arXiv:2501.14060
},
}

\bib{Manin}{article}{
   author={Manin, Ju. I.},
   title={Parabolic points and zeta functions of modular curves},
   language={Russian},
   journal={Izv. Akad. Nauk SSSR Ser. Mat.},
   volume={36},
   date={1972},
   pages={19--66},
   issn={0373-2436},
   review={\MR{0314846}},
}

\bib{MSD}{article}{
   author={Mazur, B.},
   author={Swinnerton-Dyer, P.},
   title={Arithmetic of Weil curves},
   journal={Invent. Math.},
   volume={25},
   date={1974},
   pages={1--61},
   issn={0020-9910},
   review={\MR{0354674}},
   doi={10.1007/BF01389997},
}

\bib{MY}{article}{
    AUTHOR = {Miller, Stephen D},
    author ={Yang, Tonghai},
     TITLE = {Non-vanishing of the central derivative of canonical {H}ecke
              {$L$}-functions},
   JOURNAL = {Math. Res. Lett.},
  FJOURNAL = {Mathematical Research Letters},
    VOLUME = {7},
      YEAR = {2000},
    NUMBER = {2-3},
     PAGES = {263--277},
      ISSN = {1073-2780},
   MRCLASS = {11F67 (11G05)},
  MRNUMBER = {1764321},
MRREVIEWER = {Andrea\ Mori},
       DOI = {10.4310/MRL.2000.v7.n3.a2},
       URL = {https://doi.org/10.4310/MRL.2000.v7.n3.a2},
}

\bib{MR}{article}{
    AUTHOR = {Montgomery, Hugh L.},
    author ={Rohrlich, David E.},
     TITLE = {On the {$L$}-functions of canonical {H}ecke characters of
              imaginary quadratic fields. {II}},
   JOURNAL = {Duke Math. J.},
  FJOURNAL = {Duke Mathematical Journal},
    VOLUME = {49},
      YEAR = {1982},
    NUMBER = {4},
     PAGES = {937--942},
      ISSN = {0012-7094,1547-7398},
   MRCLASS = {12A70 (10D24 14K07)},
  MRNUMBER = {683009},
MRREVIEWER = {Reinhard\ B\"olling},
       URL = {http://projecteuclid.org/euclid.dmj/1077315537},
}


\bib{MM}{article}{
   author={Murty, M. Ram},
   author={Murty, V. Kumar},
   title={Mean values of derivatives of modular $L$-series},
   journal={Ann. of Math. (2)},
   volume={133},
   date={1991},
   number={3},
   pages={447--475},
   issn={0003-486X},
   review={\MR{1109350}},
   doi={10.2307/2944316},
}





\bib{Qiu}{article}{
    author={Qiu, Congling},
    title={Finiteness properties for Shimura curves and modified diagonal cycles},
    journal={arXiv:2310.20600
},
}

\bib{QZ1}{article}{
    author={Qiu, Congling},
    author={Zhang, Wei},
    title={Vanishing results in Chow groups for the modified diagonal cycles},
    journal={arXiv:2209.09736
},
}

\bib{QZ2}{article}{
    author={Qiu, Congling},
    author={Zhang, Wei},
    title={Vanishing results for the modified diagonal cycles II: Shimura curves},
    journal={arXiv:2310.19707
},
}

\bib{Roh}{article} {
    AUTHOR = {Rohrlich, David E.},
     TITLE = {Galois conjugacy of unramified twists of {H}ecke characters},
   JOURNAL = {Duke Math. J.},
  FJOURNAL = {Duke Mathematical Journal},
    VOLUME = {47},
      YEAR = {1980},
    NUMBER = {3},
     PAGES = {695--703},
      ISSN = {0012-7094,1547-7398},
   MRCLASS = {12A70 (12A65)},
  MRNUMBER = {587174},
MRREVIEWER = {Alan\ Candiotti},
       URL = {http://projecteuclid.org/euclid.dmj/1077314189},
}

\bib{Sch}{article}{
   author={Scholl, A. J.},
   title={Classical motives},
   conference={
      title={Motives},
      address={Seattle, WA},
      date={1991},
   },
   book={
      series={Proc. Sympos. Pure Math.},
      volume={55, Part 1},
      publisher={Amer. Math. Soc., Providence, RI},
   },
   isbn={0-8218-1636-5},
   date={1994},
   pages={163--187},
   review={\MR{1265529}},
   doi={10.1090/pspum/055.1/1265529},
}

\bib{Wald}{article}{
   author={Waldspurger, J.-L.},
   title={Correspondances de Shimura},
   language={French},
   conference={
      title={Proceedings of the International Congress of Mathematicians,
      Vol. 1, 2},
      address={Warsaw},
      date={1983},
   },
   book={
      publisher={PWN, Warsaw},
   },
   isbn={83-01-05523-5},
   date={1984},
   pages={525--531},
   review={\MR{0804708}},
}

\bib{Yang}{article} {
    AUTHOR = {Yang, Tonghai},
     TITLE = {On {CM} abelian varieties over imaginary quadratic fields},
   JOURNAL = {Math. Ann.},
  FJOURNAL = {Mathematische Annalen},
    VOLUME = {329},
      YEAR = {2004},
    NUMBER = {1},
     PAGES = {87--117},
      ISSN = {0025-5831,1432-1807},
   MRCLASS = {11G10 (11G15 11G40)},
  MRNUMBER = {2052870},
MRREVIEWER = {Henri\ Darmon},
       DOI = {10.1007/s00208-004-0511-8},
       URL = {https://doi.org/10.1007/s00208-004-0511-8},
}

\bib{YY}{article}{
   author={Yang, Tonghai},
   author={Yin, Hongbo},
   title={Difference of modular functions and their CM value factorization},
   journal={Trans. Amer. Math. Soc.},
   volume={371},
   date={2019},
   number={5},
   pages={3451--3482},
   issn={0002-9947},
   review={\MR{3896118}},
   doi={10.1090/tran/7479},
}

 \bib{YZZ09}{article}{
   author={Yuan, Xinyi},
   author={Zhang, Shou-Wu},
   author={Zhang, Wei},
   title={The Gross-Kohnen-Zagier theorem over totally real fields},
   journal={Compos. Math.},
   volume={145},
   date={2009},
   number={5},
   pages={1147--1162},
   issn={0010-437X},
   review={\MR{2551992}},
   doi={10.1112/S0010437X08003734},
}
 
\bib{YZZ}{book}{
   author={Yuan, Xinyi},
   author={Zhang, Shou-Wu},
   author={Zhang, Wei},
   title={The Gross-Zagier formula on Shimura curves},
   series={Annals of Mathematics Studies},
   volume={184},
   publisher={Princeton University Press, Princeton, NJ},
   date={2013},
   pages={x+256},
   isbn={978-0-691-15592-0},
   review={\MR{3237437}},
}

\bib{Zhang}{article}{
   author={Zhang, Shou-Wu},
   title={Gross-Schoen cycles and dualising sheaves},
   journal={Invent. Math.},
   volume={179},
   date={2010},
   number={1},
   pages={1--73},
   issn={0020-9910},
   review={\MR{2563759}},
   doi={10.1007/s00222-009-0209-3},
}

	\end{biblist}
\end{bibdiv}
\end{document}